\documentclass[11pt]{amsart}
\usepackage{pdfsync}

\usepackage{amsmath, amssymb,mathtools}
\usepackage{tikz}
\usetikzlibrary{decorations.pathmorphing}

\newtheorem*{THM}{Theorem}

\usepackage{pdfsync}

\usepackage{amsmath, amssymb,mathtools}
\usepackage{tikz}
\usetikzlibrary{decorations.pathmorphing}

\DeclareMathOperator{\LBAxx}{\calL_{BA,\bbU}}

\newcommand{\cstar}{$\textrm{C}^*$}

\newcommand{\sfU}{\mathsf U} 
\newcommand{\sfX}{\mathsf X} 
\newcommand{\sfY}{\mathsf Y} 
\newcommand{\dminus}{\dot -}

\newcommand{\bbJ}{\mathbb J}

\newcommand{\bt}{\mathbf t}

\newcommand{\bbN}{{\mathbb N}}
\newcommand{\bbQ}{{\mathbb Q}}
\newcommand{\bbS}{{\mathbb S}}

\newcommand{\bbG}{{\mathbb G}}

\newcommand{\bbF}{\mathbb F}

\newcommand{\bbR}{\mathbb R}
\newcommand{\bbC}{\mathbb C}

\newcommand{\bbK}{{\mathbb K}}
\newcommand{\bbL}{{\mathbb L}}
\newcommand{\cJ}{{\mathcal J}}
\newcommand{\cK}{{\mathcal K}}
\newcommand{\calL}{\mathcal L}

\newcommand{\cZ}{{\mathcal Z}}

\newcommand{\cE}{{\mathcal E}}

\DeclareMathOperator{\id}{id}

\newcommand{\cF}{\mathcal F}

\newcommand{\e}{\varepsilon}

\newtheorem{thm}{Theorem}
\newtheorem{theorem}{Theorem}[section] %%%%%%%%%%%%%%% IF changed
\newtheorem{corollary}[theorem]{Corollary}

\newtheorem{question}[theorem]{Question}
\newtheorem{claim}[theorem]{Claim}

\newtheorem{lemma}[theorem]{Lemma}
\newtheorem{prop}[theorem]{Proposition}

\theoremstyle{definition}

\newtheorem{definition}[theorem]{Definition}

\newtheorem{example}[theorem]{Example}

\DeclareMathOperator{\Fin}{Fin}

\newcounter{my_enumerate_counter}
\newcommand{\pushcounter}{\setcounter{my_enumerate_counter}{\value{enumi}}}
\newcommand{\popcounter}{\setcounter{enumi}{\value{my_enumerate_counter}}}

\newcommand{\cP}{{\mathcal P}}
\newcommand{\cO}{{\mathcal O}}

\newcommand{\cU}{{\mathcal U}}
\newcommand{\cV}{\mathcal V}

\newcommand{\bbZ}{\mathbb Z}

\newcommand{\bbB}{\mathbb B}

\newcommand{\cI}{\mathcal I}

\DeclareMathOperator{\Ad}{Ad}

\newcommand{\twolN}{\{0,1\}^{<\bbN}}

\newcommand{\fB}{\mathfrak B}

\numberwithin{equation}{section}

\DeclareMathOperator{\Clop}{Clop}

\newcommand{\bbU}{\mathbb U}
\DeclareMathOperator{\dom}{dom}
\title{Between reduced powers and ultrapowers}

\author{Ilijas Farah}

\address{Department of Mathematics and Statistics,
York University,
4700 Keele Street,
Toronto, Ontario, Canada, M3J
1P3} 
\address{Matemati\v cki Institut SANU\\
Kneza Mihaila 36\\
11\,000 Beograd, p.p. 367\\
Serbia}
\email{email: ifarah@yorku.ca}
\urladdr{http://www.math.yorku.ca/$\sim$ifarah}

\thanks{ORCID iD https://orcid.org/0000-0001-7703-6931}
\date{\today}
\begin{document}
\begin{abstract} One of our results is a transfer principle between ultrapowers and reduced powers associated with the Fr\'echet ideal. Although motivated by the Elliott classification programme, this result applies to any axiomatizable category. We also show that there exists a nonprincipal  ultrafilter $\cU$ on $\bbN$ such that for every countable (or separable metric) structure $B$ in a countable language the quotient map from the reduced power associated with the Fr\'echet ideal onto an ultrapower  has a right inverse. While the transfer principle is proved without appealing to additional set-theoretic axioms, the conclusion of the latter theorem relies on  the Continuum Hypothesis and it is independent from the standard axioms of set theory.  We also prove that in the category of \cstar-algebras  tensoring with the \cstar-algebra of all continuous functions on the Cantor space preserves elementarity. As a side note,    neither the Jiang--Su algebra $\cZ$ nor any UHF algebra share this property. 
\end{abstract} 
\maketitle

 Ultrapowers (see \S\ref{S.Reduced} for the definitions) have been a part of the standard toolbox in logic, combinatorics, functional analysis, and algebra for decades.\footnote{In this paper, $\cU$ is a nonprincipal ultrafilter on $\bbN$. A typical application of ultrapowers in set theory is concerned with ultrapowers of  transitive models of (a large enough fragment of) ZFC, and it is desirable that these ultrapowers be well-founded. This requires~$\cU$ to be $\sigma$-closed, hence excluding nonprincipal ultrafilters on $\bbN$ and other small sets.} They are frequently used as a `magnifying glass' focused on a countable (or separable metric) structure $A$. Taking the ultrapower of $A$ results in a saturated (see \S\ref{S.Types}) elementary extension $A^\cU$ of $A$. Various asymptotic phenomena of sequences in   $A$ propagate to $A^\cU$, where they are witnessed by its \emph{elements}. This often makes the arguments more succinct and (to those familiar with the method) transparent. This method is a textbook example  of Shelah's  Mountain Air Thesis, \cite[6.4(b)]{shelah2003logical}.

        Ultrapowers were introduced independently to operator algebras, shortly before \L o\'s's introduction of ultrapowers to logic (see the first page of \cite{sherman2009notes}). The method of classifying separable operator algebras by studying their position inside an ultrapower dates back to the seminal   \cite{McDuff:Central} and \cite{Connes:Class}. McDuff's  results were adapted to \cstar-algebras in \cite{effros1978c}, with a small twist. Instead of ultrapowers, Effros and Rosenberg used the asymptotic sequence algebra $\ell_\infty(A)/c_0(A)$ (known to logicians as the reduced power associated with the Fr\'echet ideal). This algebra is at some level easier to grasp: Unlike  ultrapowers, it is canonical and its construction  does not require any form of the Axiom of Choice.  Although reduced powers (and reduced products) have been studied by logicians for decades, the absence of \L o\'s's Theorem renders them not as effective as the ultrapowers (see however    footnote~\ref{Foot.reindexing}).   
        
        The fact that ultrapowers are often interchangeable with asymptotic sequence algebras within the theory of \cstar-algebras  is quite puzzling. One of the objectives of this paper is to establish a relation between ultrapowers and reduced powers associated with the Fr\'echet ideal of countable (and separable metric) structures that explains this phenomenon. Our main result  transcends the category of \cstar-algebras and is applicable to  the category of structures in any countable language (discrete or metric). The direct motivation for this work stems from a concrete problem encountered in  Elliott's classification programme for nuclear \cstar-algebras.

%\subsection{The original motivation}\label{S.Elliott} 
%Although its results describe a basic relation between reduced powers and ultrapowers in a very abstract setting, the original motivation for this work came from Elliott's classification program 

%\subsection{Realizing morphisms} \label{S.realizing}

In 1989, Elliott conjectured that certain functor $F$ (presently known as  \emph{Elliott's invariant}) classifies a large and important class of \cstar-algebras (see~\cite{Ror:Classification}, \cite{winter2017structure}, and also  \cite{elliott2015classification} for a direct approach without using ultrapowers or reduced powers).\footnote{For a curious reader, this is the class of  nuclear, separable, unital, simple, \cstar-algebras that satisfy the Universal Coefficient Theorem (UCT) and  tensorially absorb the Jiang--Su algebra $\cZ$; but this is besides the point.} Ultrapowers are used in the classification programme of \cstar-algebras in situations  where  the simplicity of massive quotient algebra is desirable (as in the Kirchberg--Phillips classification of Kirchberg algebras, \cite{Phi:Classification}) and in the stably finite case, when one takes direct  advantage of the tracial ultrapower whose fibres are ultrapowers of~II$_1$ factors (for an excellent example of this technique see \cite{schafhauser2018subalgebras}). With the increased sophistication, one feature of the asymptotic sequence algebras not shared by the ultrapowers came to be seen as indispensable; we'll return to this in a moment. 

 A distinguishing requirement of the Elliott classification is its \emph{functoriality}: The functor  $F$ is required to have two properties, existence and uniqueness,  that we now describe in a greater generality needed later on.  

\begin{definition} \label{Def.Realized} 
If $F\colon \bbK\to \bbL$ is a functor, a morphism $\alpha\colon F(A)\to F(B)$ is \emph{realized} by a morphism $\Phi\colon A\to B$ if $F(\Phi)=\alpha$. 
More generally, given an arrow $\iota\colon B\to C$ into another object $C$, it is said that $\alpha$ is \emph{realized} by a morphism $\Phi\colon A\to C$ if $F(\iota)\circ \alpha$ is realized by $\Phi$. 
 \end{definition}

The functor $F$ satisfies   \emph{existence} if for all separable \cstar-algebras  $A$ and $B$, every morphism $\alpha\colon F(A)\to F(B)$ is realized by a morphism $\Phi\colon A\to B$.

The statement of uniqueness involves a notion of equivalence  between morphisms.  Two $^*$-homomorphisms, $\Psi_1$ and $\Psi_2$, from $A$ into $C$ are (\emph{approximately}) \emph{unitarily equivalent} if  there is a net  of unitaries $u_\lambda$ such that  $\Ad u_\lambda\circ \Psi_2$ converges  to $\Psi_1$ in the point-norm topology. 

The functor $F$ satisfies \emph{uniqueness} if all morphisms that realize the same~$\alpha$ are approximately unitarily equivalent. Since approximately unitarily equivalent morphisms share their  Elliott invariant, functorial classification requires~$F$ to be an isomorphism of categories, with the appropriately identified arrows in the domain.

 An intermediate step in proving that a morphism $\alpha\colon F(A)\to F(B)$ is realized by some $\Phi\colon A\to B$ often involves a massive extension (an ultrapower or an asymptotic sequence algebra) $C$ of $B$, with the diagonal embedding $\iota\colon B\to C$, in which one proves that $\alpha$ is realized (in the sense of the second part of Definition~\ref{Def.Realized}) by $\Phi\colon A\to C$. If $C$ is the asymptotic sequence algebra $B^\infty:=\ell^\infty(B)/c_0(B)$, finding such $\Phi$ is, together with uniqueness,  all that it takes to prove the existence. We take a moment to briefly describe the reason for this. Suppose $B$ is a metric structure. Every injection $f\colon \bbN\to \bbN$ defines an endomorphism $\Phi_f$ of $B^\infty$, by its action on the representing sequences (see \S\ref{S.Reduced}):  
\begin{equation}\label{Eq.Phif}
	\Phi_f((a_i)_{i\in \bbN})=(a_{f(i)})_{i\in \bbN}.
\end{equation}
  The following is \cite[Theorem~4.3]{gabe2019new} and (essentially)    \cite[Proposition~1.37]{Phi:Classification} (it is proven using a variant of the approximate intertwining technique).\footnote{\label{Foot.reindexing} This result comes across as a special instance of a general model-theoretic result that can be applied to metric structures with an appropriately defined and well-behaved notion of an inner automorphisms, waiting to be isolated.} 
  
\begin{THM} If $A$ and $B$ are \cstar-algebras and $A$ is   separable, then a  $^*$-homo\-morphism $\Psi\colon A\to B^\infty$  is unitarily equivalent to a $^*$-homomorphism whose range is included in (the diagonal copy of) $B$ if and only if it is approximately unitarily equivalent to $\Phi_f\circ \Psi$ for every increasing $f\colon \bbN\to \bbN$. \qed
\end{THM}

As an immediate corollary, if $F$ is a functor whose domain is the category of \cstar-algebras that satisfies uniqueness and $A$ and $B$ are separable, then (using the terminology from the second sentence of Definition~\ref{Def.Realized}) a morphism $\alpha\colon F(A)\to F(B)$ is realized by a $^*$-homomorphism $\Phi\colon A\to B$ if and only if it is realized by a $^*$-homomorphism $\Phi\colon A\to B^\infty$. 
   
 If $\cU$ is an ultrafilter and $f$ is a function from the index-set of $\cU$ into itself, then $f$ is either constant on a set of $\cU$ or it sends a set in $\cU$ to a set that does not belong to $\cU$ (e.g., \cite[Lemma~9.4.5]{Fa:STCstar}). Because of this, a map of the form $\Phi_f$  as in \eqref{Eq.Phif} is an endomorphism of an ultrapower if and only if it is the identity map, hence there is no analog of the reindexing technique for ultrapowers in place of the asymptotic sequence algebras.  Ironically, it is often easier to find $\Phi$ that realizes a morphism $\alpha\colon F(A)\to F(B)$    in an ultrapower of~$B$.

To recapitulate: A morphism $\alpha\colon F(A)\to F(B)$ can be realized by a $^*$-homomorphism from $A$ into $B^\cU$, but this is not quite as helpful as realizing $\alpha$ by a $^*$-homomorphism from $A$ into $B^\infty$. 
This `gap' begs a question, asked by Chris Schafhauser and Aaron Tikuisis and  answered by the following (for an arbitrary metric structure $B$, by $B^\infty$ we denote the reduced power $\prod_{\Fin} B$; in the case of \cstar-algebras, this reduces to $\ell^\infty(B)/c_0(B)$ as used above).

\begin{thm}
	 \label{T.AaCh} 
Suppose $F$ is a functor whose domain is a 
category $\bbK$ of structures in some countable language.  For separable  $A$ and $B$ in $\bbK$, a nonprincipal ultrafilter $\cU$ on $\bbN$,  and a  morphism $\alpha\colon F(A)\to F(B)$ 
the following are equivalent.%\footnote{Here $\iota_{B,\infty}$ denotes the diagonal embedding of $B$ into $B^\infty$ and $\iota_{B,\cU}$ denotes the diagonal embedding of $B$ into $B^\cU$, see \eqref{Eq.iota} in \S\ref{S.Reduced}.}
\begin{enumerate}
\item \label{1.C0} The morphism $\alpha$ is realized %\footnote{This is a minor modification of the  standard terminology from Elliott's program.}  
 by a morphism $\Phi\colon A\to B^\cU$
 for some (any) nonprincipal ultrafilter $\cU$ on $\bbN$. 
\item \label{2.C0} The morphism $\alpha$ is realized by a morphism $\Phi\colon A\to B^\infty$. 
\end{enumerate}
\end{thm}

Theorem~\ref{T.AaCh} and our other results are stated and proved for  structures in any countable language (discrete or metric). 
 This theorem is really a corollary of a more fundamental result, Theorem~\ref{T.5.1}. As the latter result's statement is unsuitable for inclusion in a short introduction, instead we state its consequence that uses the Continuum Hypothesis. 
  
 For a nonprincipal ultrafilter $\cU$ on $\bbN$, by $\pi_\cU\colon B^\infty \to B^\cU$ we denote the quotient map. A \emph{right inverse} to $\pi_\cU$ is a homomorphism $\Theta\colon B^\cU\to B^\infty$ such that $\pi_\cU\circ \Theta_\cU=\id_{B^\cU}$.

\begin{thm}  \label{T.A}
The Continuum Hypothesis implies that there exists a nonprincipal ultrafilter $\cU$ on $\bbN$ such that for every separable metric structure~$B$ in a separable language the quotient map $\pi_\cU\colon B^\infty\to B^\cU$ has a right inverse.   
\end{thm}

If the Continuum Hypothesis holds, then Theorem~\ref{T.AaCh} is a consequence of  Theorem~\ref{T.A}. Fortunately, the contentious issue of the true cardinality of the continuum (and a possibly even more contentious issue whether there is such a thing as a true cardinality of the continuum) can be ignored in the present context because Theorem~\ref{T.5.1} suffices for all practical purposes. Suffice it to say that  a metamathematical detour involving the Continuum Hypothesis was instrumental in finding a ZFC-result to which the (very interesting, in the author's opinion) question of the cardinality of the continuum is  completely irrelevant.   

In the case when $B$ belongs to an abelian category, the conclusion of Theorem~\ref{T.A} asserts that  the exact sequence (with $c_\cU(B)=\ker(\pi_\cU)$)
\[
0\to c_\cU(B)\to B^\infty\overset {\pi_\cU}\to  B^\cU\to 0
\]
splits. 
 Unlike $\pi_\cU$, its right inverse  is not canonical and does not necessarily exist. 
A more precise version of Theorem~\ref{T.A} involves one of the most common types of special ultrafilters on $\bbN$ (see Definition~\ref{Def.P-point}). 

\begin{thm} \label{T.split.P-point} 
Suppose that the Continuum Hypothesis holds. For a nonprincipal ultrafilter $\cU$ on $\bbN$ 
	the following are equivalent. 
	\begin{enumerate}
	\item \label{1.T.split.P-point}  For every separable metric structure $B$ in a countable language the quotient map $\pi_\cU\colon B^\infty\to B^\cU$ has a right inverse. 

\item \label{2.T.split.P-point}  $\cU$ is a P-point. 
	\end{enumerate}
\end{thm}
  
  The salient feature  of  both $ \pi_\cU$ and its right inverse, that each one of them is equal to the identity on the diagonal copies of $B$ in $B^\infty$ and $B^\cU$, is essential for the applications. It is preserved in  the following poor man's version of Theorem~\ref{T.A} that applies to all nonprincipal ultrafilters on $\bbN$.  

\begin{thm} 
\label{T.transfer} 
Suppose that the Continuum Hypothesis holds, $\cU$ is an ultrafilter on $\bbN$,  and $B$ is a  separable metric structure. Then there exists a surjective homomorphism $\Phi_\cU\colon B^\infty \to B^\cU$ that is equal to the identity on the diagonal copy of $B$ in $B^\infty$ and has a right inverse. 
\end{thm}

Theorem~\ref{T.5.1} is a variant of Theorem~\ref{T.transfer} proved without any additional set-theoretic axioms, yet sufficiently strong to imply Theorem~\ref{T.AaCh}.  Since our main results depend on both the choice of an ultrafilter and  the model of ZFC, it is unlikely that they, or even Theorem~\ref{T.AaCh}, could have been discovered without a nontrivial use of logic.

The following  (motivated by the observations given in  \S\ref{S.functor}) was pivotal in discovering our main results.

\begin{thm}\label{C.infty} There is a covariant functor $K$ from the category of metric structures into itself such that for every $B$ there is an embedding $\iota_{B,K}\to KB$ with the following property.   
	If the Continuum Hypothesis holds and if $B$ is a separable metric structure in a separable language, then there is an isomorphism $\Lambda\colon (KB)^\cU\to B^\infty$ such that $\iota_{B,\infty}=\Lambda\circ \iota_{B,K}$. 
\end{thm}

A proof of Theorem~\ref{C.infty} is given at the end of \S\ref{S.functor}.

In the case of \cstar-algebras,  the list of equivalences in Theorem~\ref{T.split.P-point} can be expanded by adding analogs to an influential result due to Sato and Kirchberg--R\o rdam (see \cite[Theorem~3.3]{KirRo:Central}, \cite[Lemma~2.1]{sato2011discrete}). If $B$ is a \cstar-subalgebra of a \cstar-algebra $C$, then the \emph{relative commutant} of $B$ inside $C$ is 
\[
C\cap B'=\{c\in C\mid bc=cb\text{ for all }b\in B\}. 
\] 
Following Kirchberg, we say that an ideal $J$ in a \cstar-algebra is a  \emph{$\sigma$-ideal} if  for every countable subset $J_0$ of $J$ there exists a positive contraction $e\in J$ such that $ea=a$ for all $a\in J_0$ (\cite{Kirc:Central},  \cite{kirchberg2014central}). 

\begin{thm} 
\label{T.C*-P-point} 	
For a P-point $\cU$ on $\bbN$ and a separable \cstar-algebra $B$ the following statements hold.\footnote{Operator algebraists should take a note that in this theorem, and elsewhere in this note,  $B^\cU$ denotes the ultrapower of a metric structure, in particular the norm ultrapower of $B$ if $B$ is a \cstar-algebra, and not  the tracial ultrapower (except in  Example~\ref{Ex.Rinfty}, where the metric structure is the hyperfinite II$_1$ factor). This choice of notation was made for consistency, and it hopefully does not lead to confusion or alienation.}   
\begin{enumerate}
\item\label{1.T.C*} $\pi_\cU[B^\infty\cap B']=B^\cU\cap \pi_\cU[B]'$. 	
\item \label{2.T.C*} Every separable \cstar-subalgebra $A$ of $B^\infty$ satisfies 
\[
\pi_\cU[B^\infty\cap A']=B^\cU\cap \pi_\cU[A]'.
\] 
\item\label{4.T.C*} The kernel of $\pi_\cU$, $c_\cU(B)$, is a $\sigma$-ideal in $B^\infty$. 
\pushcounter
\end{enumerate}
If $B$ is a UHF algebra,  the Continuum Hypothesis holds, and $\cU$ is any nonprincipal ultrafilter on $\bbN$, then \eqref{1.T.C*} %--\eqref{4.T.C*} 
is equivalent to each of the following.  
\begin{enumerate}
\popcounter
	\item\label{3.T.C*} $\cU$ is a P-point. 
	\item \label{5.T.C*} The exact sequence $0\to c_\cU(B)\to B^\infty \to B^\cU\to 0$ splits. 
\end{enumerate}
\end{thm}

Can any of  theorems~\ref{T.A}--\ref{T.C*-P-point} 	 be proven without appealing to an additional set-theoretic assumption? This depends on the theory of the structures in question (see Example~\ref{Ex.Z/2Z}).  In \cite[Theorem~D]{farah2020between} it was proven that in a forcing extension of the set-theoretic universe introduced in \cite{Fa:Embedding}, for every separable \cstar-algebra~$A$ and every ultrafilter $\cU$ on $\bbN$, $(A\otimes C(K))^\cU$ is not isomorphic~$A^\infty$ or even to a \cstar-subalgebra of $B^\infty$ for any separable \cstar-algebra $B$.   In \cite[Theorem~C]{farah2020between} the analogous result, in the same forcing extension,  was proven for models that satisfy  the order property as witnessed by a quantifier-free formula (this includes \cstar-algebras and II$_1$ factors). Thus in this forcing extension the conclusion of each of theorems~\ref{T.A}, \ref{T.transfer}, and \ref{C.infty} fails for models of many unstable theories. In this model the real line cannot be covered by fewer than $2^{\aleph_0}$ meager sets, and therefore in it P-points, and even selective ultrafilters, exist (this is well-known, see e.g., \cite[Proposition~8.5.7]{Fa:STCstar}). This implies that  the conclusion of Theorem~\ref{T.split.P-point}  and the equivalence of \eqref{3.T.C*} and \eqref{5.T.C*} in Theorem~\ref{T.C*-P-point} also fail in this model.

One more thing. The question  which \cstar-algebras  have the property that taking the minimal tensor product with them preserves elementary equivalence was raised in \cite[Question~3.10.5]{Muenster}. In \S\ref{S.el.tensor} we give a positive answer  for $C(K)$ ($K$ denotes the Cantor space) and a negative answer for $\cZ$ (the Jiang--Su algebra) and  all UHF algebras.

Our results are proved by model-theoretic and set-theoretic analysis  of the structure $(B^\infty,B^\cU,\pi_\cU)$  (see~\S\ref{S.Term}).    In~\S\ref{S.FV} we compute the theory of a structure of this sort (a `Feferman--Vaught-type' result).   In \S\ref{S.functor} we define a functor $K$  such that $B^\infty$ is elementarily equivalent to $KB$ for every  metric structure~$B$. In \S\ref{S.CS}  we prove that $(B^\infty,B^\cU,\pi_\cU)$ is countably saturated when~$\cU$ is a P-point.  The proofs of Theorems~\ref{T.AaCh}, \ref{T.A}, \ref{T.split.P-point},  and \ref{T.transfer} can be found in  \S\ref{S.Proofs}. In \S\ref{S.el.tensor} we prove that in the category of \cstar-algebras tensoring with $C(K)$ preserves elementarity and tensoring with the Jiang--Su algebra or a UHF algebra does not.  Theorem~\ref{T.C*-P-point} 	 is proved in \S\ref{S.Sato}, and we conclude with a few general remarks in \S\ref{S.Concluding}.    All proofs are given in the case of metric structures but they yield proofs in the case of   classical, discrete, structures--simply ignore the epsilons and the deltas.

\subsection*{Acknowledgments} I would like to thank  
 Aaron Tikuisis and  Chris Schaf\-hauser for raising the question answered in Theorem~\ref{T.AaCh} and for many stimulating discussions. The results of the present paper greatly extend those of (now obsolete) \cite{farah2019between}, where the case of Theorem~\ref{T.A} in which $B$ is a separable, unital \cstar-algebra was proven. This was extended to the non-unital case in a conversation with Aaron Tikuisis in Oberwolfach, August 2019. 
 The special case of Theorem~\ref{T.AaCh} in the case when $\bbK$ is the category of \cstar-algebras and~$F$ is the Elliott invariant, total $K$-theory, or the algebraic $K_1$, was proved by Chris Schafhauser by a delicate argument (unpublished). I am indebted to Andreas Thom for asking whether in the conclusion of Theorem~\ref{T.transfer}  one can require $\Phi_\cU$ to be the quotient map---i.e., whether the conclusion of  Theorem~\ref{T.A} holds. 
I would also like to thank Udi Hrushovski for pointing out to a very misleading bit of notation in an earlier draft of this note and to Bojana Laskovi\'c for finding some typos.   Last, but not least, I am grateful to the anonymous referee for an exceptionally helpful report.  
 
Preliminary versions of the results from this note were presented during 2019 in several talks and two mini-courses   (one at the YM\cstar A/YW\cstar A conference at the University of Copenhagen in August, and another at the Mathematical Institute of the Serbian Academy of Sciences and Arts in October and November). I would like to thank to both institutions, and to Zoran Petri\'c and Bori\v sa Kuzeljevi\'c in particular, for warm hospitality and stimulating discussions. 

\section{Preliminaries} 

\label{S.Term} We follow the standard model-theoretic terminology, as exposed  in \cite{Mark:Model} or~\cite{ChaKe}. For model theory of  metric structures and \cstar-algebras see \cite{BYBHU} and \cite{Muenster}, respectively, as well as \cite[\S 16]{Fa:STCstar}. In this logic, the interpretations of formulas are $\bbR$-valued and the analog of the Lindenbaum boolean algebra of a theory is a real Banach algebra.  All of our results are stated in model theory of metric structures. They specialize to classical model theory by considering discrete structures as structures in $\{0,1\}$-metric. The special discrete variants of our results, although novel, nontrivial, and with a potential for applications, will not be stated explicitly. Readers not interested in continuous logic can omit all arguments that involve approximating elements of a model up to some~$\e>0$ and still obtain complete proofs. 

By $\varphi^A(\bar a)$ we denote the interpretation of a formula $\varphi(\bar x)$ at a tuple $\bar a$ of the same sort as $\bar x$ in the structure $A$.

\subsection{Reduced powers}\label{S.Reduced} 
 Given a metric language $\calL$, an infinite indexed family of $\calL$-structures $C_j$, for $j\in \bbJ$,\footnote{The only index-set used in our main results is $\bbJ=\bbN$, but some of the intermediate (yet quotable) results hold in larger generality.} 
 and an ideal $\cJ$ on $\bbJ$, the \emph{reduced product}  is the quotient of 
 $\prod_j C_j$ corresponding to the pseudometric 
 \[
 d(a, b)=\inf_{\sfX\in \cJ} \sup_{j\in \bbJ\setminus \sfX} d(a_j, b_j).
 \]  
 If $\calL$ is a discrete language, then the domain of the reduced product  associated with $\cJ$ is the set of equivalence classes on $\prod_{j\in \bbJ} C_j$  where $a$ and $b$  are identified if and only if the set  $\{j\in \bbJ\mid a_j\neq b_j\}$ belongs to $\cJ$. 

 In the case when $\calL$ is a multisorted language, this is used to define every sort of the reduced product. All function symbols are interpreted in the natural way, pointwise, and the relation symbols $R$  are interpreted by ($\bar a$ denotes a tuple of the appropriate sort)
 \[
 R^{\prod_\cJ C_j}(\bar a):=\inf_{\sfX\in \cJ} \sup_{j\in \bbJ\setminus \sfX} R^{M_j}(\bar a_j)
 \]  
 in the metric case. 
 In the classical, discrete, case, we define the interpretation of $R$ by setting $R^{\prod_\cJ C_j}(\bar a)$ to hold if and only if $\{j\in \bbJ\mid R(\bar a_j)$ fails$\}\in \cJ$.

The reduced product associated to $\cJ$ is denoted $\prod_\cJ C_j$.   In the case when all $C_j$ are equal to $B$ we write $B^\cJ$ for $\prod_\cJ B$.
 The two extremal cases are most important. 
 If $\cJ$ is the Fr\'echet ideal, denoted $\Fin$, then $B^\cJ$ is denoted~$B^\infty$, and if $\cJ$ is a maximal ideal then $B^\cJ$ is denoted~$B^\cU$ (where~$\cU$ is the complement of $\cJ$) and called \emph{ultrapower}. If $\cU$ is disjoint from $\cJ$, then we have a natural quotient map  $\pi_\cU\colon B^\cJ\to B^\cU$. 
 
 An element $b$ of  $B^\cJ$  is determined by a representing sequence $(b_i)_{i\in \bbJ}$. We routinely commit the crime of identifying the elements of $B^\cJ$ with the corresponding representing sequences in $\prod_{i\in \bbJ} B_i$ (the latter is considered as a sorted product, where sort $S$ is interpreted as $\prod_{i\in \bbJ} S(B_i)$, where $S(B_i)$ is the interpretation of sort $S$ in $B_i$). This practice is similar to the analogous well-established practice in the case of $L_p$ spaces: It is innocuous, as long as one knows what they are doing.  
For an ideal $\cJ$ we define the diagonal embedding
\begin{equation}
\label{Eq.iota} 
\iota_{B,\cJ}\colon B\to B^\cJ 
\end{equation}
by sending $b$ to the constant representing sequence $(b,b,\dots)$ and 
 identify $B$ with its diagonal image in $B^\cJ$. We write $\iota_{B,\infty}$ for $\iota_{B,\Fin}$. 
By $\bar a$ we denote a tuple  $(a_0, \dots, a_{n-1})$ of an unspecified length (but   `of the appropriate sort', which depends on the context). 
The arity of $\bar a$ will be routinely suppressed, and we will write $\bar a\in B$ for $\bar a\in B^n$ where $n$ is the arity of $\bar a$. 
For the sake of brevity,   variables  are sometimes omitted
and a formula $\varphi(\bar x)$ is written as $\varphi$. 
When dealing with tuples of representing sequences, in order to avoid confusion 
the entries of $\bar a$ will be  denoted $a(0), \dots, a(n-1)$. 
If~$\bar a$ is an $n$-tuple of elements of $B^\cJ$, then $a_j(i)$, for $j\in \bbJ$, is a representing sequence of the $i$-th entry of~$\bar a$.

\subsection{Types and saturation} \label{S.Types} We recall definitions of a condition and a type. Fix a metric language $\calL$. A \emph{condition} is an expression $\varphi(\bar x)=r$, where $\varphi(\bar x)$ is an $\calL$-formula and $r$ is a real number. It is \emph{satisfied} by a tuple $\bar a$ of an appropriate sort in an $\calL$-structure $A$ if $\varphi^A(\bar a)=r$  An \emph{$n$-type} is a set of conditions in variables $x_i$, for $i<n$. A \emph{type} is an $n$-type for some $n$. (Types in variables $\bar x$ can be construed as functionals on the Lindenbaum algebra of all $\calL$-formulas all of whose free variables are among $\bar x$, see \cite[\S 4.7]{Muenster} or \cite[\S 16.1]{Fa:STCstar}; although very intuitive, this formulation will not be used explicitly.)  A type $\bt(\bar x)$ is \emph{realized} by a tuple $\bar a$ in an $\calL$-structure $A$ if each of its conditions is satisfied by $\bar a$. It is \emph{satisfiable} in an $\calL$-structure $A$ if for every finite list of conditions $\varphi_i(\bar x)=r_i$, for $i<n$, in $\bt(\bar x)$ and every $\e>0$ some $\bar a$ in $A$  satisfies $\max_{i<n}|\varphi_i^A(\bar a)-r_i|<\e$.  

A type is countable if it is countable as a set of conditions. An $\calL$-structure~$A$ is \emph{countably saturated} if every countable, satisfiable, type in the language obtained by expanding $\calL$ by adding constants for the elements of $A$ is realized in the expansion of $A$ to this language obtained by interpreting these constants in the natural way.\footnote{In \cite{shelah1990classification}, such models are called countably compact. However, for countable (or separable) languages our definition of counfeferman
table saturation coincides with Shelah's.} Both ultrapowers and reduced products associated with the Fr\'echet ideal are countably saturated (see \cite{ChaKe} for the discrete logic, in the metric case see \cite{BYBHU} for ultrapowers and \cite{FaSh:Rigidity} for the reduced powers associated with the Fr\'echet ideal, or see \cite[\S 16]{Fa:STCstar} for both).

\subsection{P-points} 
The space of ultrafilters on $\bbN$ is naturally identified with the \v Cech--Stone compactification $\beta\bbN$ of $\bbN$. Then the space of non-principal ultrafilters is identified with the remainder (corona), $\beta\bbN\setminus \bbN$. 

\begin{definition} \label{Def.P-point} 
An ultrafilter $\cU$ on $\bbN$ is a \emph{P-point} if for every sequence $X_n\in \cU$, for $n\in \bbN$, there exists $X\in \cU$ such that $X\setminus X_n$ is finite for all $n$. 
\end{definition}

Equivalently, $\cU$ is a P-point if the intersection of every countable family of open neighbourhoods of $\cU$ includes an open neighbourhood of $\cU$.   
Not every ultrafilter on $\bbN$ is a P-point: 
By compactness, any countable subset $X$ of $\beta\bbN\setminus \bbN$ has a nonempty set of accumulation points, and none of its members is  a P-point.   
By a classical result of W. Rudin, the Continuum Hypothesis implies that P-points exist (\cite{Ru}). It is  relatively consistent with ZFC that there are no P-points in $\beta\bbN\setminus \bbN$ at all (see \cite[\S 6.4]{Sh:PIF} and \cite{chodounsky2019there}).

\section{Theories of reduced powers} 
\label{S.FV}

In this section we compute the theory of the structure $(B^\infty,B^\cU,\pi_\cU)$ from the theory of $B$. 
Towards this goal, we prove a strengthening of Ghasemi's Feferman--Vaught theorem (\cite{ghasemi2016reduced}). It applies to structures of the form 
\[
(B^\cI,B^\cJ,\pi_\cJ)
\]
where $B$ is an $\calL$-structure, $\cI$ and $\cJ$ are ideals on the index-set $\bbJ$ such that $\cI\subseteq \cJ$, and $\pi_\cJ\colon B^\cI\to B^\cJ$ 
is the quotient map.
Let $\calL^2$ denote the language of this structure. It has two `meta-sorts' corresponding to $B^\cI$ and $B^\cJ$, and each $\calL$-sort corresponds to one sort in each of the two meta-sorts. It also has the function symbol $\pi_\cI$ interpreted as the quotient map. 

In order to state our main result, let 
 $\LBAxx$ be the language of Boolean algebras with the  Boolean operations  $\wedge,\vee, {}^\complement, 0$, and $1$, also equipped with 
constants~$Z^{\zeta}_{t}$ for every $\calL^2$-formula $\zeta$
 and every $t\in \bbQ$, and with an additional predicate $\bbU$. 
The following  is modelled on  the  eponymous notion from  \cite{ghasemi2016reduced} and  \cite[Definition~16.3.2]{Fa:STCstar}. 

\begin{definition} \label{Def.1/k.determined}
For $k\geq 2$ and language $\calL$, an $\calL^2$-formula $\varphi(\bar x)$ is \emph{$k$-determined}\footnote{In \cite{ghasemi2014reduced} the definition of  `$k$-determined' is different, but in both cases this notion is used only within the proof of a theorem.} 
if objects with the following properties  exist. 
 \begin{enumerate}
\item  A finite set $\bbF[\varphi,k]$ of $\calL$-formulas whose free variables are included in the free variables of $\varphi(\bar x)$, 
\item  $\LBAxx$-sentences
$\theta^{\varphi,k}_{l}$, for $0\leq l\leq k$, such that the following conditions hold. 
\begin{enumerate}
\item 	All  constants appearing in  $\theta^{\varphi,k}_l$  are among $Z^{\zeta}_t$, for $\zeta\in \bbF[\varphi,k]$, $t\in \bbQ\cap [0,1]$. 
\item Each $\theta^{\varphi,k}_l$ is \emph{increasing}, 
i.e., if $\bar X=(X_i)$ are its free variables  and $A_i\leq A_i'$ for all $i$ are elements of a Boolean algebra $\bbB$,  then 
$(\theta^{\varphi,k}_l)^\bbB(\bar A)$ implies~$(\theta^{\varphi,k}_l)^\bbB(\bar A')$.    
\end{enumerate} 
\pushcounter
 \end{enumerate}
 These objects are required to satisfy the following. 
  Given     $\calL$-structures  $(M_j)_{j\in \bbJ}$, and ideals $\cI\subseteq \cJ$ on $\bbJ$, 
for   $\zeta(\bar x)\in \bbF[\varphi,k]$, 
we write~$\theta^{\varphi,k}_l[\bar a]$ 
for the value of $\theta^{\varphi,k}_l$ in the quotient Boolean algebra $\cP(\bbJ)/\cI$ 
with ($[\sfX]_{\cI}$ denotes the equivalence class of~$\sfX$ modulo $\cI$)\footnote{Note that the value of $Z^\zeta_{l/k}$ depends on $\bar a$.}
\[
Z^{\zeta}_{l/k}:=[\{j: (\zeta(\bar a_j))^{M_j}> l/k\}]_\cI
\]
and with $\bbU(Z^\zeta_{l/k})$ true if and only if $Z^\zeta_{l/k}\notin \cJ$. 
 Then   the following holds (writing $M$ for the structure 
 $(\prod M_j/\cI, \prod_j M_j/\cJ, \pi_\cJ)$) 
\begin{enumerate}
\popcounter
\item\label{I.theta.l.1} $\varphi^M(\bar a)> (l+1)/k$ implies  
$ \theta^{\varphi,k}_l[\bar a]$
and 
\item\label{I.theta.l.2}  $\theta^{\varphi,k}_l[\bar a]$
implies $\varphi^M(\bar a)> (l-1)/k$. 
\end{enumerate}
\end{definition}
Definition~\ref{Def.1/k.determined} asserts that  the value of $\varphi^M(\bar a)$  is determined up to $2/k$  by 
a finite set of formulas $\theta^{\varphi,k}_{l}$ for 
$0\leq l\leq k$, which are in turn determined by the evaluations of the formulas in the finite set $\bbF[\varphi,k]$ in every~$M_j$.

\begin{theorem} \label{T.FV++}
For every metric language $\calL$ and every $k\geq 2$,   every $\calL^2$-formula is $k$-determined. 
\end{theorem} 

\begin{proof} The proof is analogous to the proofs of \cite[Theorem~3.1]{ghasemi2016reduced} and \cite[Theorem~16.3.3]{Fa:STCstar}, where it was proven that the formulas in a reduced product of metric structures are $k$-determined (with an appropriate definition). We will follow the template of the latter proof and indicate only the necessary changes. The proof proceeds by finding the required objects and demonstrating that they are as required in the case of  
an arbitrary  $(M_j)_{j\in \bbJ}$, ideals $\cI$  and~$\cJ$ on $\bbJ$ such that $\cI\subseteq \cJ$,  with 
\[
\textstyle \tilde M:=(\prod_j M_j/\cI,\prod_j M_j/\cJ, \pi_\cJ), 
\]
 and $\bar a\in \tilde M$ of the same sort as $\bar x$ in the formula $\varphi(\bar x)$ being considered.
 
  By induction on complexity of the formula  $\varphi$,  it suffices to prove that
the set of all $k$-determined formulas satisfies the following closure properties:
\begin{enumerate}
\item\label{1.FV}  All atomic formulas are $k$-determined. 
\item \label{2.FV} If $\varphi$ is $k$-determined, so is $\frac 12 \varphi$. 
\item If $\varphi$ and $\psi$ are $2k$-determined, then  $\varphi\dminus \psi$ is $k$-determined. 
\item \label{4.FV} If $\varphi$ is $k$-determined, so are $\sup_x\varphi$ and $\inf_x\varphi$
for every variable $x$. 
\end{enumerate}
Only the treatment of case \eqref{1.FV} is different from that in \cite[Theorem~16.3.3]{Fa:STCstar},  and it splits into two cases.

In the first case, the terms of the atomic formula $\varphi$  are evaluated in the meta-sort corresponding to $\prod_j M_j/\cI$. 
Then let 
$\bbF[\varphi,k]:=\{\varphi\}$    and let~$\theta^{\varphi,k}_l$ be the formula $Z^{\varphi}_{l/k}\neq 0$. 
This formula is clearly  increasing. 
Since 
\[
\textstyle \varphi(\bar a)^{M}=\limsup_{j\to \cI} \varphi^{M_j}(\bar a_j), 
\]
we have that  $\varphi(\bar a)^{M}>(l+1)/k$ 
implies $\theta^{\varphi,k}_l[\bar a]$.  Similarly,  $\theta^{\varphi,k}_l[\bar a]$ 
implies 
$
\varphi(\bar a)^{M}>(l-1)/k$, as required.

In the second case, the terms of the atomic formula $\varphi$ are evaluated in the meta-sort corresponding to $\prod_j M_J/\cJ$. 
Then let 
$\bbF[\varphi,k]:=\{\varphi\}$    and let~$\theta^{\varphi,k}_l$ be the formula $\bbU(Z^{\varphi}_{l/k})$. 
This formula is clearly  increasing. Also,  
\[
\textstyle \varphi(\bar a)^{M}=\limsup_{j\to \cJ} \varphi^{M_j}(\bar a_j), 
\]
and since $\bbU(Z)$ if and only if $[Z]_\cJ\neq 0$, the proof proceeds as in the first case. 

  The treatment of \eqref{2.FV}--\eqref{4.FV} is  identical to that in \cite[Theorem~16.3.3]{Fa:STCstar}, and therefore omitted. 
  This completes the inductive  proof. 
 \end{proof}

In the proof of countable saturation (Theorem~\ref{T.CS}) we will need the following theorem that applies to an arbitrary language $\calL$ and is stated in the terminology introduced at the beginning of Section~\ref{S.FV}. 

\begin{corollary}
	 \label{C.FV}  For every  finite set $\bbG$ of  $\calL^2$-formulas\footnote{By adding dummy variables, we may assume that all formulas in $\bbG$ have the same tuple of free variables $\bar x$.} and every $\e>0$ there are a finite  set $\bbF[\bbG,\e]$ of $\calL$-formulas and $\delta>0$ with the following property.  
 If $(M_j)_{j\in \bbJ}$  are  $\calL$-structures and  $\cI$ and $\cJ$  are ideals on  $\bbJ$ such that $\cI\subseteq \cJ$,  
 then with $M=(\prod_j M_j/\cI,\prod_j M_j/\cJ, \pi_\cJ)$ 
 for all  $\bar a_j\in M_j$ and $\bar b_j\in M_j$  of the appropriate sort we have that  
\[
\max_{\zeta\in \bbF[\varphi,\e]}  \limsup_{j\to \cI} |\zeta^{M_j}(\bar a_j)-\zeta^{M_j}(\bar b_j)|<\delta
\]
 implies
 $
\max_{\varphi\in \bbG}| \varphi^M(\bar a)-\varphi^M(\bar b)|<\e
$.
\end{corollary}
The displayed formula does not refer to $\cJ$, but since $\cJ\supseteq \cI$ no information is lost and this is not a problem. 

\begin{proof} Choose $k>2/\e$ and $\delta=1/k$. Then the conclusion 
 follows from the fact that each $\varphi\in \bbG$ is $k$-determined and the finiteness of $\bbG$. 
\end{proof}

\section{The functor $K$ and elementary embeddings}
\label{S.functor}

For the definitions of conditions,   types, and saturation see \S\ref{S.Types}. 
The structure $B^\infty$ is countably saturated (\cite{FaSh:Rigidity}, see also \cite[\S 16.5]{Fa:STCstar}). If~$B$ is separable and the Continuum Hypothesis holds, then $B^\infty$ has both cardinality and density character equal to $\aleph_1$, and it is  therefore saturated. Since any two elementarily equivalent saturated structures are isomorphic (this is a classical result of Keisler, see e.g., \cite{BYBHU} or \cite[\S 16]{Fa:STCstar} for the continuous variant),  $B^\infty$ is isomorphic to the ultrapower of any one of its separable elementary submodels $C$, via an isomorphism that extends the identity map on $C$. This observation begs the following question ($\iota_{B,\infty}\colon B\to B^\infty$ is the diagonal embedding):  

\begin{question} 
Given a countable language $\calL$, is there a functor $K$ from the category of separable $\calL$-structures into itself such that for every $B$ there exist an  embedding $\iota_{B,K}\colon  B\to  KB$ and an elementary embedding  $\Psi\colon KB\to  B^\infty$  such that  $\iota_{B,\infty}=\Psi\circ \iota_{B,K}$? 
\end{question}

A positive answer is given in \S\ref{S.FunctorK}, giving a satisfactory explanation to the fact that $B^\infty$ behaves as an ultrapower. It involves a canonical construction of a separable substructure $KB$ of $B^\infty$  for every separable $B$ such that the ultrapower of $KB$ is isomorphic to $B^\infty$ via an isomorphism that commutes with the diagonal embeddings of $B$ if the Continuum Hypothesis holds (Theorem~\ref{C.infty}).\footnote{Caveat: the embedding $\Psi$ is, unlike $\iota_{B,\infty}$ and $\iota_{B,K}$, not canonical.}  

\subsection{The functor $K$}\label{S.FunctorK}
Given a metric structure $A$ and a compact Hausdorff space $X$, let
 \[
 C(X,A)=\{f\colon X\to A: f\text{ is continuous}\}. 
 \]
This construction is functorial, because a morphism $f\colon A\to B$ 
defines a morphism from $C(X,A)$ to $C(X,B)$ that sends $g$ to  $f\circ g$. 
Let $\iota_{B,K}\colon B\to KB$ be the diagonal embedding that sends $b\in B$ to the corresponding constant function. We will identify $B$ with its diagonal image in~$KB$. 
  In the case when~$X$ is the Cantor space, $K$, we denote this functor by $K$ and in particular write $KA$ for $C(K,A)$.  
 
If $X$ is a topological space and $A$ is a metric structure, 
then a function   $f\colon K\to A$ is \emph{locally constant}
if there exists a partition of $K$  into clopen sets $K=\bigsqcup_{i<m} U_i$ such that the restriction of $f$ to $U_i$ is constant for every $i$.  
 
\begin{example}\label{Ex.1} 
\begin{enumerate}
\item \label{Ex.0} For every structure $A$ define an inductive system of  structures  $A_n$, for $n\in \bbN$, by $A_0=A$ and  $A_{n+1}=A_n\oplus A_n$ for all $n$, with the connecting maps $a\mapsto (a,a)$. The inductive limit $\lim_n A_n$  is isomorphic to the substructure of $KA$ consisting of  locally constant functions.   
\item If $A$ is discrete then $KA$ consists of locally constant functions as described in \eqref{Ex.0}. 
\item 	\label{Ex.1.2}
If $A$ is a metric structure, then $KA$ is the completion of the structure of all locally constant functions from $K$ into $A$ with respect to the uniform metric, $d_\infty(f,g):=\max_{x\in K} d(f(x),g(x))$. 

\item If $A$ is a \cstar-algebra then $KA\cong A\otimes C(K)$, where the isomorphism respects the diagonal copies of $A$ and $C(K)$. This is a well-known general fact about \cstar-algebras.   One way to see it is to note that the algebraic tensor product, $A\odot C(K)$, is isomorphic to the structure whose elements are locally constant functions as described in~\eqref{Ex.0}. Since $C(K)$ is abelian the algebraic tensor product has the unique \cstar-norm, with the completion isomorphic to  $A\otimes C(K)$. 
\end{enumerate}
\end{example}

Lemma~\ref{L.transfer} below provides a template for the conclusion of Theorem~\ref{T.transfer}. The condition that both maps be 1-Lipshitz is important because in logic of metric structures every function symbol is equipped with a modulus of uniform continuity (1-Lipshitz maps are also known as contractions in the theory of  operator algebras). 

\begin{lemma} \label{L.transfer} For any structure $B$ there is a surjective homomorphism $\pi_0\colon KB\to B$ such that  $\iota_{B,K}\colon B\to KB$ is the  right inverse of $\pi_0$, $\pi_0$ is equal to the identity on the diagonal copy of $B$ in $KB$, and both $\pi_0$ and $\iota_{B,K}$ are 1-Lipshitz. 
\end{lemma}

\begin{proof} Fix a point in $K$, denoted 0 and let $\pi_0$ be the evaluation map at 0. Then $\pi_0$ and $\iota_{B,K}$  clearly have the required properties. 
\end{proof}

The fact that the right inverse $\pi_0$ depends on the choice of a point in $K$ will come back to haunt us. This lack of canonicity of $\pi_0$ is one of the reasons why the proof of Theorem~\ref{T.A} uses the Continuum Hypothesis.

\begin{prop} \label{P.embedding} If $B$ is a metric structure and $\cI$ is an ideal on $\bbN$ such that the Boolean algebra $\cP(\bbN)/\cI$ is atomless then there is an embedding $\Psi\colon KB\to B^\cI$ such that  $\iota_{B,\cI}=\Psi\circ  \iota_{B,K}$.
\end{prop}

\begin{proof} A subset of $\bbN$ is called \emph{$\cI$-positive} if it does not belong to $\cI$.  
Since $\cP(\bbN)/\cI$ is an atomless Boolean algebra, we can recursively find $\cI$-positive sets $X_s$, for $s\in \twolN$, with the following properties for all $s$. 
\begin{enumerate}
\item 	 $X_{\langle\rangle}=\bbN$. 
\item Each $X_s$ is equal to the disjoint union  $X_{s^\frown 0}\sqcup X_{s^\frown 1}$. 
\pushcounter
\end{enumerate}
We continue the proof by elaborating Example~\ref{Ex.1} \eqref{Ex.0}. 
For $n\geq 1$ let 
\begin{enumerate}
\popcounter
\item \label{Eq.KnB}
$K_n B:=\{a\in B^\infty: (\forall s\in \{0,1\}^n) (\exists a(s)\in A)(\forall j\in X_s) a_j=a(s)\}$. 
\end{enumerate}
Then $K_n B\cong B^{2^n}$ and $K_n B\subseteq K_{n+1} B$ for all $n$. 

A well-known back-and forth argument shows that all countable atomless Boolean algebras are isomorphic (e.g., \cite[Proposition~1.4.5]{ChaKe:1973}). 
Since the Boolean algebra $\Clop(K)$ of clopen subsets of $K$ is countable and atomless, it is isomorphic to the subalgebra of $\cP(\bbN)/\cI$ generated by $X_s$, for $s\in \twolN$.  By fixing an isomorphism, we can identify $K_n B$ as defined in \eqref{Eq.KnB} with a substructure of $B^\infty$. These identifications are isometric and compatible with one another. 
By Example~\ref{Ex.1} \eqref{Ex.1.2}, $\bigcup_n K_n B$ is dense in $KB$, and we have an isometric isomorphism between $KB$ and the closure of $\bigcup_n K_n B$ inside $B^\cJ$. This isometric isomorphism is the required embedding $\Psi$. 
\end{proof} 

%In the following $\pi_0$ is the evaluation map as defined in Lemma~\ref{L.transfer}.

\begin{prop} \label{P.elementary} Suppose that  $\calL$, $B$, $\cI$, and $\Psi$ are  as in Proposition~\ref{P.embedding}.   
\begin{enumerate}
\item\label{1.P.elementary} The embedding $\Psi\colon KB\to B^\cI$ is elementary and $\iota_{B,\cI}=\Psi\circ \iota_{B,K}$. 
\item \label{2.P.elementary} If $\cU$ is an ultrafilter disjoint from $\cI$, $\pi_\cU\colon B^\cI\to B^\cU$ is the quotient map, and $\pi_0\colon KA\to A$ is the evaluation map as in Lemma~\ref{L.transfer},  then the embedding~$\Psi$ can be chosen so that the embedding 
\[
(\Psi,\iota_{B,\cI})\colon (KB,B,\pi_0)\to (B^\cI, B^\cU, \pi_\cU)
\]
 is elementary.  
\end{enumerate}
\end{prop}

\begin{proof} 
\eqref{1.P.elementary}
We will identify $KB$ with its image under $\Psi$. 
 By the Tarski--Vaught test (\cite[Theorem~2.6.1]{Muenster}) it suffices to prove that if  
	$\varphi(\bar x, y)$ is a formula  and  $\bar a$ in $KB$ is a tuple of the 
	appropriate sorts, then\footnote{All variables are sorted, thus in $\inf_{y}$ the variable $y$ ranges over the appropriate sort~$S$. In the case of \cstar-algebras, each sort is an  $n$-ball of the \cstar-algebra for some fixed $n$, hence $y\in S(KB)$ is equivalent to `$y\in KB$ and $\|y\|\leq n$.'} 
\[
\inf_{y\in S(KB)}\varphi^{B^\infty}(\bar a,y)=
\inf_{y\in S(B^\infty)}\varphi^{B^\infty}(\bar a,y). 
\]
This is  equivalent to asserting that for every $\e>0$  and every $d\in B^\infty$ 
there exists $c\in KB$ such that 
$
\varphi^{B^\infty}(\bar a, c)< \varphi^{KB}(\bar a,d)+\e$. 

Since the locally constant functions are dense in $KB$ (Example~\ref{Ex.1} \eqref{Ex.1.2}) by possibly changing $\e$ we can assume that every entry of $\bar a$ is locally constant (i.e., a step-function from $K$ into $B$). Let $\sfX_s$, for $s\in \twolN$, be the subsets of $\bbN$ constructed in the proof of Proposition~\ref{P.embedding} used to define $\Psi$.  Fix an $m$ large enough so that every entry of $\bar a$ is constant on  $\sfX_s$ for every $s\in \{0,1\}^m$.  

By Corollary~\ref{C.FV} there are $n\geq 1$, $\delta>0$, and $\calL$-formulas $\zeta_i(\bar x,y)$, for $i<n$,  with the same free variables as $\varphi$, such that for all $\bar b,\bar b',c$, and $c'$  in~$B^\infty$ of the appropriate sorts,  $\limsup_j \max_{i<n} |\zeta_i^{\tilde B}(\bar b_j,c_j)-\zeta_i^{\tilde B}(\bar b'_j,c'_j)|\leq \delta$
implies $|\varphi^{B^\infty}(\bar b,c)-\varphi^{B^\infty}(\bar b',c')|<\e$.

For $j\in \bbN$ define an element of $\bbR^n$ by  
\[
\bar r(j)=((\zeta_0^B(\bar a_j,\bar d_j), \dots, \zeta_{n-1}^B(\bar a_j, \bar d_j))). 
\]
 Since every formula has bounded codomain, the set of all $\bar r(j)$ is included in a sufficiently large open neighbourhood of 0 in $\bbR^n$ with a compact closure. We can therefore find a partition $\bigsqcup_{i<m'} \sfY_i$ refining  $\bigsqcup_{s\in \{0,1\}^m} \sfX_s$ such that for every $i<m'$ and all $j$ and $j'$ in $\sfY_i$ we have $\max_{i<m} |\bar r(j)_i-\bar r(j')_i|<\delta$.  
For every $j<m'$ there exists a unique $s(j)\in \{0,1\}^m$ such that $\sfY_j\subseteq \sfX_{s(j)}$.  Also, $\sfX_s=\bigcup_{s(j)=s} \sfY_j$ for all $s\in \{0,1\}^m$. 
 
 Let $l$ be such that $2^l> m'$. The sets $\sfX_t$, for $t\in \{0,1\}^{m+l}$, form a partition of $\bbN$ into infinite sets  that refines the partition $\sfX_t$, for $t\in \{0,1\}^m$.  For every $s\in \{0,1\}^m$, the set $\{j<m'\vert \sfY_j\subseteq \sfX_s\}$ has fewer elements than the set $\{t\in \{0,1\}^{m+l}\vert \sfX_t\subseteq \sfX_s\}$.  Therefore there is a surjection $\{0,1\}^{m+l}\to m'$, $t\mapsto j(t)$, such that $s(j(t))$ is an initial segment of~$t$ for every~$t$. 

 We can therefore choose  a permutation~$\chi$ of  $\bbN$ with the following properties. 
\begin{enumerate}
\popcounter
\item [(a)] $\chi[\sfX_s]=\sfX_s$, for all $s\in \{0,1\}^m$. 
\item [(b)] $\chi[\sfY_j]=\sqcup\{\sfX_t\vert t\in \{0,1\}^{m+l}, j(t)=j\}$. 
\pushcounter	
\end{enumerate}
Every  permutation of $\bbN$ naturally acts on $B^\bbN$, by sending $(b_j)$ to $(b_{\chi(j)})$. Since the Fr\'echet ideal is invariant under permutations,~$\chi$ defines an automorphism, denoted  $\Phi_\chi$, of $B^\infty$. Clearly such `permutation automorphism' fixes the diagonal copy of $B$ pointwise. 
  Since $\Phi_\chi[\sfX_s]=\sfX_s$ for all $s\in \{0,1\}^m$,we have  $\Phi_\chi(\bar a)=\bar a$.

For each $j<m'$ choose $k(j)\in \sfY_j$. Define $c\in B^\infty$ by its representing sequence, 
\begin{enumerate}
\popcounter 
\item [(c)]	$c_i:= d_{k(j(t))}$, if $t\in \{0,1\}^{m+l}$ is such that $i\in \sfX_t$. 
\end{enumerate}
Then $d(d_{\chi(i)},c_i)<\delta$ for all $i\in \bbN$, hence $d(\Phi_\chi(d),c)\leq \delta$.  
Also,  $c$ is constant on each $\sfX_t$, for $t\in \{0,1\}^{m+l}$, and it therefore belongs to $\Psi[KB]$. 
By the choice of the sets $\sfY_j$ we have 
\[
\max_{k<m}|\zeta_k^B(\bar a,  c)-\zeta_k^B(\bar a, d)| \leq \delta, 
\] 
and the choice of $\delta$ implies 
$
\varphi^{B^\infty}(\bar a, c)<\varphi^{B^\infty}(\bar a, d)+\e
$ 
as required. 

 \eqref{2.P.elementary} Unlike the proof of \eqref{1.P.elementary},  in this proof we revisit the construction of~$\Psi$ from Proposition~\ref{P.embedding} and choose the sets $\sfX_s$, for $s\in \twolN$ with  additional care.  Since $\cU$ is an ultrafilter, for every partition of a  $\cU$-positive set into two pieces exactly one of the pieces is $\cU$-positive.  Since $\bbN=X_{\langle\rangle}$ is $\cU$-positive, in the construction of $\Psi$ we can  choose the sets $X_s$ so that $X_s\in \cU$ implies $X_{s^\frown 0}\in \cU$ for all $s\in \twolN$, hence  $X_{s}\in \cU$ if and only if $s(i)=0$ for all $i\in \dom(s)$. Therefore $\Psi(\pi_0(f))=\pi_\cU(\Psi(f))$ for all $f\in KB$,  and $(\Psi,\iota_{B,\cI})$ is an embedding of $(KB,B)$ into $(B^\cI, B^\cU)$
 
 The proof that this embedding is elementary is virtually identical to the proof  of \eqref{1.P.elementary}. The only extra care needs to be taken in the choice of the permutation~$\chi$. Since $\cU$ is an ultrafilter and the sets $\sfY_i$, for $i<m'$, form a partition of $\bbN$, there is a unique $i<m'$ such that $\sfY_i\in \cU$.  The permutation~$\chi$ needs to satisfy $\chi[\sfY_i]\in \cU$. In other words, with $l$ and $j(t)<m'$  as in the proof of  \eqref{1.P.elementary},  $\chi$ is chosen so that  $\chi[\sfY_{i}]\supseteq \sfX_t$ for the unique $t\in \{0,1\}^{m+l}$ such that  $\sfX_t\in \cU$ and $j(t)=i$. Therefore $\Phi_\chi$ defines an automorphism of $(B^\infty, B^\cU,\pi_\cU)$ and the proof proceeds as in \eqref{1.P.elementary}.  The other details  are omitted. 
 \end{proof}

The following application will be reformulated in \S\ref{S.el.tensor}. 

\begin{corollary}\label{C.K-preservation} Suppose that $B$ and $C$ are structures in the same language. 
\begin{enumerate}
\item \label{1.C.K} If $B$ and  $C$ are elementarily equivalent,  then so are $KB$ and $KC$. 
\item \label{2.C.K} If $\Psi\colon B\to C$ is elementary, then $K\Psi\colon KB\to KC$ is elementary. 
\end{enumerate}
\end{corollary}

\begin{proof}  
By Proposition~\ref{P.elementary}, $KB$ 
is isomorphic to an elementary submodel of $B^\infty$
and $KC$ is isomorphic to an elementary submodel of $C^\infty$. By \cite[Proposition~3.6]{ghasemi2016reduced}  (or the results of \S\ref{S.FV}), the operation of taking reduced product over $\Fin$  preserves elementary equivalence, and  therefore $B^\infty$ and $C^\infty$ are elementarily equivalent and so are $KB$ and $KC$ by transitivity. This proves \eqref{1.C.K}. To prove \eqref{2.C.K}, note that   if $\Psi\colon B\to C$ is an elementary embedding then $\Psi^\infty\colon B^\infty\to C^\infty$ is an elementary embedding by Theorem~\ref{T.FV++}. Therefore $K\Psi$ is an elementary embedding of $KB$ into~$KC$. 
\end{proof}

\begin{proof}[Proof of Theorem~\ref{C.infty}] Suppose the Continuum Hypothesis holds and fix a separable metric structure $B$ in a separable language. We need to find an isomorphism between $(KB)^\cU$ and $B^\infty$ that commutes with the diagonal embeddings of $B$. By Proposition~\ref{P.elementary}, there is an elementary embedding $\Psi\colon KB\to B^\infty$ such that  $\iota_{B,\infty}=\Psi\circ \iota_{B,K}$. Since $(KB)^\cU$ and $B^\infty$ are both countably saturated (\cite[Proposition~16.4.2]{Fa:STCstar} and \cite[Theorem 16.5.1]{Fa:STCstar}, respectively) and of cardinality $\aleph_1=2^{\aleph_0}$, they are both saturated. Since $B$ is separable,  $\Psi$ can be extended to an isomorphism  $\Lambda\colon (KB)^\cU\to B^\infty$ (\cite[Theorem~16.7.5]{Fa:STCstar}). 
\end{proof}

\section{Countable saturation of composite quotients}
\label{S.CS} 

%%%%

Theorem~\ref{T.CS} is the main technical result of this section (see \S\ref{S.Types} for the explanation of the terminology). Its proof uses Proposition~\ref{P.QE} which must have been known to Tarski. Since I could not find a reference to this result, a sketch of its proof is included.  

\begin{prop}\label{P.QE}
	The theory $T_\sfU$ of atomless Boolean algebras with an additional unary predicate $\sfU$ for an ultrafilter admits elimination of quantifiers. 	
\end{prop}

\begin{proof} 
The easiest way to prove this may be to show that if $A$ and $B$ are countable models of $T_{\sfU}$ and $\Phi_0\colon A_0\to B_0$ is an isomorphism between finitely generated submodels of $A$ and $B$, then $\Phi_0$ extends to an isomorphism between $A$ and~$B$.   
Fix such $A$ and $B$ and an isomorphism $\Phi_0\colon A_0\to B_0$. 

To find $\Phi$, note that a finitely generated submodel of a model of $T_\sfU$ is a finite Boolean algebra with a distinguished ultrafilter. This ultrafilter is uniquely determined by the atom that belongs to it. Write $A$ and $B$ as increasing unions of finite Boolean algebras, $A=\bigcup_j A_j$ and $B=\bigcup_j B_j$ such that $|A_j|=|B_j|$ for all $j$. By recursion find isomorphisms $\Phi_j\colon A_j \to B_j$ such that $\Phi_{j+1}$ extends $\Phi_j$ for all $j$. This sequence uniquely determines an isomorphism $\Phi\colon A\to B$ that extends $\Phi_0$. 

This shows that in a countable model of $T_\sfU$ the type of a finite subset is uniquely determined by its quantifier-free type. Since the language is countable this holds in an arbitrary model of $T_\sfU$ and completes the proof. 
\end{proof}

\begin{theorem}\label{T.CS}
For every structure $B$, if $\cU$ is a P-point ultrafilter on $\bbN$ then 
the structure $(B^\infty,B^\cU,\pi_\cU)$ is countably saturated. 
\end{theorem}

\begin{proof} Fix $B$,  $\cU$, and a countable satisfiable type $\bt(\bar x)$ over 
\[
\tilde B:=(B^\infty,B^\cU,\pi_\cU).
\] 
 Enumerate $\bt(\bar x)$ as a sequence of conditions
$\varphi_i(\bar x)=s_i$, for $i\in \bbN$. By composing $\varphi_i$ with a piecewise linear function we may  assume that its range  is included in $[0,1]$  for all $i$. 
By   Corollary~\ref{C.FV}, for each $k\geq 1$ 
there exist a finite set of formulas $\bbF(k)$  and $d(k)\geq 1$ 
such that for all  $\bar a$ and $\bar b$  in $\tilde B$ 
 of the appropriate sort, 
\[
\max_{\zeta\in \bbF(k)}\limsup_j |\zeta^{B}(\bar a_j)-\zeta^{B}(\bar b_j)|<1/d(k)
\] 
implies $\max_{j\leq k} |\varphi_j^{\tilde B}(\bar a)-\varphi_j^{\tilde B}(\bar b)|<1/k$. 

We may assume that the codomain of every $\zeta\in \bbF$ is  $[0,1]$ and that
 $\bbF(k)\subseteq \bbF(k+1)$ for all $k$. 
Let $\bbF:=\bigcup_k \bbF(k)$.   
For $\zeta\in \bbF$ and $\bar a$ in $\tilde B$ of the appropriate sort  let  
\[
Z^\zeta_{t}(\bar a):=\{n: \zeta^{\tilde B} (\bar a_n)>t\}. 
\]
For  $k\in \bbN$ and $S\subseteq d(k)\times \bbF(k)$,  let (we declare hitherto undefined sets to be empty): 
\[
\textstyle \Phi_{S,k}(\bar a):=\bigcap_{(j,\zeta)\in S} (Z^\zeta_{j/d(k)}(\bar a)\setminus Z^\zeta_{(j+1)/d(k)}(\bar a)). 
\]
Also let 
\[
\textstyle  \Upsilon_{k}(\bar a):=\{S\subseteq d(k)\times \bbF(k): \Phi_{S,k}(\bar a) \text{ is finite}\}
\]
and 
\[
\textstyle  \Upsilon_{k,\cU}(\bar a):=\{S\subseteq d(k)\times \bbF(k): \Phi_{S,k}(\bar a) \notin \cU\}. 
\]
From the choice of $\bbF(k)$ and $d(k)$, we have the following. 

\begin{claim} \label{C.2.3} 
For $\bar a\in B$, the sets $\Upsilon_k(\bar a)$ and $\Upsilon_{k,\cU}(\bar a)$ determine the value of $\varphi_j^{\tilde B}(\bar a)$ up to $1/k$. 
\end{claim} 

\begin{proof} 
If the codomain of $\zeta$ is  $[0,1]$, then the sets $Z^\zeta_{j/d(k)}\setminus Z^\zeta_{(j+1)/d(k)}$, for $j\leq d(k)$, form 
a partition of $\bbN$. Therefore $\Upsilon_k(\bar a)$ and $\Upsilon_{k,\cU}(\bar a)$, for $k\in \bbN$, together determine the isomorphism type of the finite Boolean algebra generated by $Z^\zeta_{j/d(k)}$, for $\zeta\in \bbF(k)$ and $0\leq j\leq d(k)$ with a distinguished ultrafilter $\cU$.  By Proposition~\ref{P.QE}, the theory of atomless Boolean algebras with a distinguished ultrafilter admits elimination of quantifiers. Therefore the type of the tuple $Z^\zeta_{j/d(k)}$, for $j\leq d(k)$, in the Boolean algebra $\cP(\bbN)/\Fin$ with the distinguished ultrafilter $\cU$ is determined by its quantifier-free type.  By the choice of $\bbF(k)$,  this determines the value of $\varphi_j^{\tilde B}(\bar a)$ for all $j$. 
\end{proof} 

Since  the type $\bt(\bar x)$ is satisfiable, for every $k$ there exists $\bar b(k)$ of the appropriate sort in $B$ 
such that 
$\max_{i<k}|\varphi_i^{\tilde B}(\bar b(k))-s_i|<1/k$. 
Let $\cV$ be a nonprincipal ultrafilter on $\bbN$ (we could use $\cU$,   but we want to emphasize that the role of $\cV$ is different from that of $\cU$). 
Since for a fixed $k$ there are only finitely many possibilities for the sets $\Upsilon_k(\bar a)$ and $\Upsilon_{k,\cU}(\bar a)$, 
there exist $\Upsilon^*_{k}$ and
$\Upsilon^*_{k,\cU}$  such that 
\begin{align*} 
\{n\in \bbN: \Upsilon_{k}(\bar b(n))=\Upsilon^*_{k}\text{ and } 
\Upsilon_{k}(\bar b(n))=\Upsilon^*_{k}\}\in \cV. 
\end{align*}

\begin{claim}  
If $\Upsilon_{k}(\bar b)=\Upsilon^*_{k}$ and $\Upsilon_{k,\cU}(\bar b)=\Upsilon_{k,\cU}^*$  for all $k$ then $\varphi_i^{\tilde B}(\bar b)=s_i$ for all $i$ and $\bar b$ satisfies the type $\bt(\bar x)$. 
\end{claim} 

\begin{proof}
Fix $j$. By Claim~\ref{C.2.3}, we have $\varphi_j^{\tilde B}(\bar b) =\lim_{k\to \cV} \varphi^{\tilde B}_j(\bar b(k))$,  and the latter limit is equal to $s_j$.    
\end{proof}

Since $\cV$ is an ultrafilter,  we have 
\begin{enumerate}
\item \label{Eq.Ups}
$\Upsilon^*_k=\Upsilon^*_{k'}\cap (d(k)\times \bbF(k))$ 
for all $k<k'$, and 
\item \label{Eq.Ups.1}
$\Upsilon^*_{k,\cU}=\Upsilon^*_{k',\cU}\cap (d(k)\times \bbF(k))$ 
for all $k<k'$. 
\pushcounter
\end{enumerate}
By refining the sequence $\{\bar b(k)\}$, we may assume that 
$\Upsilon_k(\bar b(k))=\Upsilon_k^*$ and 
$\Upsilon_{k,\cU}(\bar b(k))=\Upsilon_{k,\cU}^*$ for all $k$. We don't need $\cV$ anymore. (Clearly we did not really need an ultrafilter, but as we have already assumed that a nonprincipal ultrafilter on $\bbN$ existed this was an easy route towards a proof of the simultaneous variant of the Bolzano--Weierstrass theorem for countably many bounded sequences.)

The set $\bbF:=\bigcup_k \bbF(k)$ is countable, and 
\[
\bbS:=\{\Phi_{S,k}(\bar b(j)): j\in \bbN,k\in \bbN, S\notin \Upsilon_{k,\cU}(\bar b(j))\}
\]
is a countable subset of $\cU$. Since $\cU$ is a P-point, we can fix $\sfX\in \cU$ such that $\sfX\setminus \sfY$ is finite for all $\sfY\in \bbS$. 
We will choose a finite  $\sfY_k\subseteq \bbN$, for $2\leq k$, 
such that for all~$k$ the following conditions hold. 
\begin{enumerate}
\popcounter
\item\label{1.Phi} $j\in \bigcup_{i\leq j} \sfY_i$ and $\sfY_i\cap \sfY_j=\emptyset$ if $i\neq j$. 
\item\label{2.Phi} $|\Phi_{S,k}(\bar b(k))\cap \sfY_k|\geq k$ for all $S\subseteq d(k)\times \bbF(k)$ such that $S\notin \Upsilon^*_{k}$.  
\item\label{3.Phi}  $\Phi_{S,k}(\bar b(k+1))\subseteq \bigcup_{j\leq k} \sfY_j$ 
for all $S\subseteq d(k)\times \bbF(k)$ such that   $S\in \Upsilon^*_{k}$. 
%\item\label{4.Phi} $|\Phi_{S,k}(\bar b(k))\cap \sfY_k|\geq k$ for all $S\subseteq d(k)\times \bbF(k)$ such that $S\notin \Upsilon^*_{k}$, and 
\item\label{4.Phi}  $\Phi_{S,k}(\bar b(k+1))\cap \sfX\subseteq \bigcup_{j\leq k} \sfY_j$ 
for all $S\subseteq d(k)\times \bbF(k)$ such that   $S\in \Upsilon^*_{k,\cU}$. 
\pushcounter
\end{enumerate}
We proceed to  describe the recursive construction of the sequence $(\sfY_k)$. It is analogous to the corresponding construction in the proof that the reduced product of metric structures associated to the Fr\'echet filter is countably saturated (\cite[Theorem~16.4]{Fa:STCstar}). 
 For   $S\notin \Upsilon^*_{2}$ the set  $\Phi_{S,2}(\bar b(2))$ is infinite 
and we have $|\Phi_{S,2}(\bar b(2))\cap  m(S)|\geq 2$ for a large enough $m(S)$.  Also, if  $S\in \Upsilon_{2,\cU}^*$ then $\Phi_{2,\cU}(\bar b(2))\notin \bbS$, hence $\Phi_{S,\cU}(\bar b(2))\cap \sfX$ is finite. We can therefore choose $m>2$ large enough so that the set (identifying $m$ with $\{0,\dots, m-1\}$)
\begin{multline*}
\textstyle \sfY_2:=m\cup \bigcup\{\Phi_{S,3}(\bar b(n(3)): S\subseteq d(3)\times \bbF(3), 
S\in \Upsilon^*_{3}\}\\
\cup\bigcup\{\Phi_{S,3}(\bar b(n(3))\cap \sfX\mid S\subseteq d(3)\times \bbF(3), S\in \Upsilon_{3,\cU}^*\}
\end{multline*}
satisfies \eqref{1.Phi}, \eqref{2.Phi},  \eqref{3.Phi}, and \eqref{4.Phi} with $k=2$. 

Suppose that $k\geq 2$ and the sets $\sfY_2,\dots, \sfY_k$ have been  chosen to satisfy 
\eqref{1.Phi}, \eqref{2.Phi},   \eqref{3.Phi}, and~\eqref{4.Phi}. 
The set  $\Phi_{S,k+1}(\bar b(k+1))$
 is infinite for every $S\notin \Upsilon^*_{k+1}$, and 
 any large enough~$m$ satisfies 
$|m\cap (\Phi_{S,k+1}(\bar b(k+1))\setminus \bigcup_{j\leq k} \sfY_j)|\geq k+1$
for all $S\notin \Upsilon^*_{k+1}$. 
Also, the set $\Phi_{S,k+2}(\bar b(k+2))$ is finite for every $S\in \Upsilon^*_{k+1}$, and the set
$\Phi_{S,k+2}(\bar b(k+2))\cap \sfX$ is finite for every $S\in \Upsilon^*_{k+1,\cU}$.  Let $m\geq k+1$ be sufficiently large. Then the set 
\begin{align*}
\textstyle\sfY_{k+1}:= &(m\setminus \bigcup_{j\leq k} \sfY_j)\\ 
&\cup  \bigcup\{\Phi_{S,k+2}(\bar b(k+2): S\subseteq d(k+2)\times \bbF(k+2), 
S\in \Upsilon^*_{k+2}\}\\
&\cup (\bigcup\{\Phi_{S,k+2}(\bar b(k+2))\cap \sfX:S\subseteq d(k+2)\times \bbF(k+2), S\in\Upsilon^*_{k+2,\cU}\} 
\end{align*}
is  finite and it satisfies \eqref{1.Phi},  \eqref{2.Phi},    \eqref{3.Phi}, and \eqref{4.Phi} with $k+1$ replacing~$k$. 

This describes the recursive construction. 
The sets $\sfY_k$, for $k\in \bbN$, are disjoint and by \eqref{3.Phi} their union includes  $\bbN$, hence they form a   partition of~$\bbN$ into finite sets. 
Define  $\bar a\in M$ by its representing sequence 
\[
\bar a_j:=\bar b(n(k))_j\quad\text{ if }j\in \sfY_k.
\] 
Then $\Phi_{S,k}(\bar a)\cap \sfY_k=\Phi_{S,k}(\bar b(n(k))$ for all $k$. 
To prove that  $\bar a$ realizes $\bt(\bar x)$,  
it suffices to prove  $\Upsilon_k(\bar a)=\Upsilon^*_k$ and $\Upsilon_{k,\cU}(\bar a)=\Upsilon_{k,\cU}^*$ for all $k$. 
Fix~$k$ and $S\subseteq d(k)^2\times \bar m(k)$. 

If  $S\notin \Upsilon^*_k$, then \eqref{2.Phi} implies  $|\Phi_{S,k}(\bar a)|\geq j$ 
for all $j\geq k$. Therefore $\Phi_{S,k}(\bar a)$ is infinite, hence  $S$ is an infinite set   
 in $\Upsilon_k(\bar a)$. 

Now suppose $S\in \Upsilon^*_k$.  Then  $S\in \Upsilon^*_j$ for all $j\geq k$. 
Fix any $l\geq k$. 
Then \eqref{3.Phi} and \eqref{Eq.Ups} together imply  $\Phi_{S,l}(\bar a)\cap   \sfY_j=\emptyset$ for all $j\geq l+1$. Hence
 $\Phi_{S,l}(\bar a)$ is  finite, and since $l\geq k$ was arbitrary,  
$S\notin \Upsilon_k(\bar a)$.  

If $S\in \Upsilon^*_{k,\cU}$, then $S\in \Upsilon^*_{j,\cU}$ for all $j\geq k$. Then $\Phi_{S,k}(\bar a)\cap \sfX\cap \sfY_j=\emptyset$ for all $j\geq k$, and since $\sfX\in \cU$ we have  $\Psi_{S,k}(\bar a)\notin \cU$. 

If $S\notin \Upsilon^*_{k,\cU}$, then $S\notin \Upsilon^*_{j,\cU}$ for all $j\geq  k$ and $\sfY_j\cap \sfX\subseteq \Phi_{S,k}(\bar a)$ for all $j\geq k$. Since $\sfX\in \cU$ Since the set $\sfX\setminus\bigcup_{j\geq k} \sfY_j$ is finite,   $\sfX\setminus \Phi_{S,k}(\bar a)$ is also finite and $\Phi_{S,k}(\bar a)\in \cU$. 

This implies that $\bar a$ satisfies the type $\bt$ in $B$. 
\end{proof}

If $\cU$ is not a P-point and $B$ contains an infinite linear order definable by a quantifier-free formula, then $(B^\infty,B^\cU,\pi_\cU)$ is not countably saturated. This simple observation is basis for the proofs of Theorem~\ref{T.split.P-point} \eqref{1.T.split.P-point} $\Rightarrow$  \eqref{2.T.split.P-point} and Theorem~\ref{T.C*-P-point} \eqref{1.T.C*}~$\Rightarrow$~\eqref{3.T.C*}. 

\section{Proofs of theorems~\ref{T.split.P-point}, \ref{T.A}, \ref{T.5.1}, \ref{T.transfer}, and \ref{T.AaCh} (in this order)}
\label{S.Proofs} 

The proofs referred to in the title amount to little more than putting together the results of the previous sections. 

\begin{proof}[Proof of Theorem~\ref{T.split.P-point}]
We first prove that if $B$ is a separable metric structure in a separable language and $\cU$ is a P-point then  the quotient map $\pi_\cU\colon B^\infty \to B^\cU$ has a right inverse.

By Proposition~\ref{P.elementary} \eqref{2.P.elementary}, the structure $\tilde B_0:=(KB,B, \pi_0)$ is isomorphic to an elementary submodel of $\tilde B:=(B^\infty,B^\cU, \pi_\cU)$.  
By Theorem~\ref{T.CS}  and the Continuum Hypothesis,  $\tilde B$ is  saturated. Since $\tilde B_0$ is separable, there is an isomorphism 
$\Phi$ between the ultrapower\footnote{The fact that we use the same ultrafilter $\cU$ is unimportant.} $(\tilde B_0)^\cU$ and $\tilde B$ that extends the identity map on $\tilde B_0$. The isomorphism $\Phi$ sends the ultrapower of the right inverse,~$\Theta$, of $\pi_0$ (Lemma~\ref{L.transfer})  to a map $\Theta_\cU\colon B^\cU\to B^\infty$. Since all the properties required from $\Theta$---being a homomorphism, right inverse to~$\pi_\cU$, and extending $\id_B$---are elementary, \L o\'s's Theorem implies that $\Theta$ is as required. 

For the converse, suppose that $\cU$ is not a P-point, and fix $X_n\in \cU$ for $n\in \bbN$ such that for every $X\in \cU$ the set $X\setminus X_n$ is infinite for some $n\in \bbN$.  Let $B$ be $(\bbN, \min)$ where $\min$ is the obvious binary function. Define  $b\in B^\infty$ by its representing sequence, $b_j=n$ if $j\in X_{n+1}\setminus X_n$ (taking $X_{-1}=\bbN$). 

Assume for the sake of obtaining a contradiction that $\Theta\colon B^\cU\to B^\infty$ is a right inverse to $\pi_\cU$ and let $c=\Theta(b)$. The set 
$
X=\{j\mid c_j=b_j\}
$
 belongs to $\cU$ because  $\pi_\cU(c)=b$.  Fix the least $n$ such that $X\setminus X_n$ is infinite. Identifying $n+1\in B$ with its diagonal image, we have $\min(b,n+1)=n+1$, but $\min(\Theta(b),n+1)=\min(c,n+1)\neq n+1$; contradiction.   
 \end{proof}

\begin{proof}[Proof of Theorem~\ref{T.A}]
Since the Continuum Hypothesis implies that a P-point exists (\cite{Ru}) this is an immediate consequence of Theorem~\ref{T.split.P-point}. 
\end{proof}

Even if $\cU$ is not a P-point and the Continuum Hypothesis fails, $\pi_\cU$ may have a right inverse and the exact sequence $0\to c_\cU(B)\to B^\infty\to B^\cU\to 0$ may split for some choices of $B$. 
If the structure $B$ is finite then $B\cong B^\cU$, and if the language of $B$ has no function symbols then any selector for $\pi_\cU$ splits the exact sequence as in Theorem~\ref{T.A}. The following example is more interesting.

\begin{example}\label{Ex.Z/2Z} Let $\cU$ be any nonprincipal ultrafilter on $\cU$. Let  $B$ be $\bigoplus_\bbN \bbZ/2\bbZ$, considered as a vector space over $\bbZ/2\bbZ$. It is a structure in the language equipped with the symbol for the addition and symbols for multiplication by scalars in $\bbZ/2\bbZ$.  Model-theorists will notice that the theory of $B$ is stable, equal to the theory of $KB$, and that it admits the elimination of quantifiers. Even without using this observation, one sees that both $B^\infty$ and $B^\cU$ are $\bbZ/2\bbZ$-vector spaces of dimension $2^{\aleph_0}$, and  transfinite recursion shows that the exact sequence $0\to c_\cU(B)\to B^\infty \to B^\cU\to 0$ splits. 
\end{example}

The instances of Theorem~\ref{T.AaCh} when $\bbK$ is the category of \cstar-algebras and  $F$ is any of the standard $K$-theoretic functors (in addition to $K_0$ and $K_1$, this includes the algebraic $K$-theory, $K$-theory with coefficients, KK,  and KL; see e.g.,  \cite{Ror:Classification}, \cite{blackadar1998k}) follow from Theorem~\ref{T.transfer} (proved below) by  the standard metamathematical absoluteness arguments (similar to e.g., \cite[Appendix 2]{adams1991amenability}). The following variant of Theorem~\ref{T.transfer} is strong enough to imply the conclusion of  Theorem~\ref{T.AaCh} for \emph{any} functor $F$ as well as Theorem~\ref{T.transfer} itself. 

\begin{theorem}
 \label{T.5.1} 
	Suppose $B$ is a separable metric structure in a separable language~$\calL$. 
	Then there exists a family $\cF$ of quadruples $(C,D,\pi,\Theta)$ with the following properties. 
	\begin{enumerate}
	\item\label{1.bandf} $C$ is a separable substructure of $B^\infty$. 
	\item $D$ is a separable substructure of $B^\cU$. 
	\item \label{3.5.1} $\pi\colon C\to D$ is a surjective homomorphism. 
	\item \label{4.5.1} $\Theta\colon D\to C$ is a homomorphism such that $\pi\circ \Theta=\id_D$. 
	\item\label{6.bandf} ($\cF$ is $\sigma$-closed) In the ordering on $\cF$ defined by $(C_1,D_1, \pi_1, \Theta_1)\leq (C_2,D_2, \pi_2,\Theta_2)$ if $C_1\subseteq C_2$, $D_1\subseteq D_2$, $\pi_2$ extends $\pi_1$, and $\Theta_2$ extends~$\Theta_1$, for every countable increasing sequence $(C_n,D_n, \pi_n, \Theta_n)$, for $n\in \bbN$, in $\cF$ we have $(\bigcup_n C_n, \bigcup_n D_n, \bigcup_n \pi_n, \bigcup_n \Theta_n)\in \cF$.\footnote{Functions are identified with their graphs, giving a meaning to the formulas $\bigcup_n \pi_n$ and $\bigcup_n \Theta_n$.}
	\item \label{7.bandf} ($\cF$ is dense) For every $(C,D,\pi,\Theta)\in \cF$, every $c\in B^\infty$, and every $d\in B^\cU$ some  $(C_1,D_1,\pi_1,\Theta_1)\in \cF$  such that $c\in C_1$ and $d\in D_1$ extends $(C,D,\pi,\Theta)$. 
	\end{enumerate}
Also, for every separable ubstructures $C_0$ of $B^\infty$ and $D_0$ of $B^\cU$, there exists a quadruple $(C, D, \pi,\Theta)$ such that $C_0\subseteq C$ and $D_\subseteq D$.  
\end{theorem}

A family $\cF$ with the properties \eqref{1.bandf}--\eqref{6.bandf}  is called a $\sigma$-complete back-and-forth system in \cite[Definition~8.2.8]{Fa:STCstar}.

\begin{proof}
		Fix   a separable metric structure $B$ in a separable language $\calL$. 
	%Consider the structure $(KB,B,\pi_0,\Theta)$ as defined in the proof of Theorem~\ref{T.A}. 
	The  structure $(KB,B)$ is by Proposition~\ref{P.elementary} elementarily equivalent to $(B^\infty,B^\cU)$. The latter structure is countably saturated since both of its components $B^\infty$ and $B^\cU$ are countably saturated and there is no relation between them. Therefore for any $\cV\in \beta\bbN\setminus \bbN$ the structures $\fB_1:=(KB,B)^\cV$ and $\fB_2:=(B^\infty,B^\cU)$ are countably saturated and elementarily equivalent, and there exists a $\sigma$-complete back-and-forth system $\cE$ of partial isomorphisms between separable substructures of $\fB_1$ and $\fB_2$ (see e.g., \cite[Proposition~16.6.4]{Fa:STCstar}). By Lemma~\ref{L.transfer}, there is a surjective homomorphism $\pi_0\colon KB\to B$ such that  $\iota_{B,K}\colon B\to KB$ is the  right inverse of $\pi_0$, $\pi_0$ is equal to the identity on the diagonal copy of $B$ in $KB$, and both $\pi_0$ and $\iota_{B,K}$ are 1-Lipshitz.   Consider the expansion of $(KB,B)$ to $(KB,B,\pi_0,\iota_{B,K})$ (this is a metric structure because both maps are 1-Lipshitz). Its ultrapower is an expansion of $\fB_1$,    $(KB,B,\pi_0, \iota_{B,K})^\cV$. The  properties required in \eqref{3.5.1} and \eqref{4.5.1} of $\pi$ and $\Theta$ are first-order and satisfied by $\pi_0$ and $\iota_{B,K}$. By \L o\'s's Theorem, they are  shared by $\pi_0^\cU$ and $\iota_{B,K}^\cU$.   The  partial isomorphisms in $\cE$ transfer these maps to maps with the  required properties.   

The last sentence is an easy consequence of \eqref{6.bandf} and \eqref{7.bandf}. 
\end{proof}

\begin{proof}[Proof of Theorem~\ref{T.transfer}] This is an immediate consequence of Theorem~\ref{T.5.1} and a standard back-and-forth construction of length $\aleph_1$ using the Continuum Hypothesis (use e.g.,  \cite[Proposition~16.6.1]{Fa:STCstar}). 
	\end{proof}

\begin{proof}[Proof of Theorem~\ref{T.AaCh}] Fix separable metric structures $A$ and $B$ and a morphism $\alpha\colon F(A)\to F(B)$. If  $\Psi\colon A\to B^\infty$  realizes $\iota_{B,\infty}\circ \alpha$, then $\pi_\cU\circ \Psi$ realizes $\iota_{B,\cU}\circ \alpha$. Now assume $\Psi\colon A\to B^\cU$ realizes $\iota_{B,\cU}\circ \alpha$. Since $\Psi$ is continuous, $\Psi[A]$ is separable and by Theorem~\ref{T.5.1}, there is a quadruple $(C,D,\pi,\Theta)$ such that $C$ is a separable substructure of $B^\infty$ that includes $B$, $D$ is a separable substructure of $B^\cU$ that includes $B\cup \Psi[A]$, $\pi\colon C\to D$ is a surjective homomorphism, $\Theta\colon D\to C$ is a homomorphism such that $\pi\circ \Theta=\id_D$, and $\pi$ and $\Theta$ commute with the diagonal embeddings of $B$.  
 Then $\Theta\circ \Psi$ realizes $\iota_{B,\infty}\circ \alpha$. 
\end{proof}

%%%%%%%%%

\section{Preservation of elementarity by tensor products of \cstar-algebras}
\label{S.el.tensor}
Some familiarity with the basic theory of \cstar-algebras is required in this section (obviously).  The question which operations on structures  preserve elementarity  was raised in \cite{feferman1959first}, where  preservation results for generalized products were proven. Reduced products of metric structures preserve elementarity by \cite{ghasemi2016reduced} and \S\ref{S.FV}. Tensor products of modules do not preserve elementarity (\cite{Olin:Direct}). The question whether tensoring with any infinite-dimen\-sio\-nal \cstar-algebra preserves elementary equivalence was asked in \cite[Question~3.10.5]{Muenster}.\footnote{This question was asked for the minimal (spatial) tensor product, but since in every example considered in \cite{Muenster} and in this section at least one of the factors is nuclear, in all examples under consideration $A\otimes B$ is necessarily the minimal tensor product.} In addition to the intrinsic interest in preservation results, this was motivated by the general problem of the extent of definability of $K$-groups in  \cstar-algebras (see \cite[\S 3.11 and \S 3.12]{Muenster} for some specific results). In the unital case the first step towards constructing $K_0$ and $K_1$ is tensoring with the algebra $\cK(H)$ of compact operators, and it is not known whether tensoring with $\cK(H)$ preserves elementarity. Tensoring with $C([0,1])$ does not necessarily preserve elementary equivalence This was proved in \cite[Proposition~3.10.3]{Muenster} as an application of \cite[Theorem~3.1]{Phi:Exponential} and the `range of invariant' results from Elliott's classification programme.  Since $C(X)\otimes A$ is isomorphic to $C(X,A)$ and ($K$ denotes the Cantor space) $K\Psi$ is $\Psi\otimes 1_{C(K)}$,  Corollary~\ref{C.K-preservation} has the following consequence.

\begin{corollary} \label{C.elementary} 
\begin{enumerate}
\item If $B$ and $C$ are elementarily equivalent \cstar-algebras, 
then so are $B\otimes C(K)$ and $C\otimes C(K)$.
\item If $\Psi\colon B\to C$ is an elementary embedding, then $\Psi\circ \id_{C(K)}$ is an elementary embedding of $B\otimes C(K)$ into $C\otimes C(K)$. \qed 
\end{enumerate}
\end{corollary}

  The following result, all but proven in \cite[\S 3.5]{Muenster}, 
 transpired during a conversation with Chris Schafhauser at the Fields Institute in 2019, and it is included with his kind permission ($\cZ$ denotes the Jiang--Su algebra, \cite{JiangSu}).  

\begin{prop} There are elementarily equivalent  \cstar-algebras $A$ and $B$ such that 
$A\otimes D$ and $B\otimes D$ are not elementarily equivalent if $D$ is $\cZ$ or  
any UHF algebra $D$. 
\end{prop}

\begin{proof} The reader is assumed to be familiar with \cite{Muenster}.  Let $A$ be the unital, monotracial, 
  \cstar-algebra constructed in \cite[Theorem 1.4]{Rob:Nuclear}
  such that $A^\cU$ does not have a unique trace and let $B=A^\cU$. 
  Then $A$ and $B$ are elementarily equivalent, $A\otimes \cZ$ is monotracial,  
  and $B\otimes \cZ$ is not. 
  As the Cuntz--Pedersen nullset in $A\otimes \cZ$
  is definable (this was essentially proved in \cite{Oza:Dixmier}, see \cite[Theorem~3.5.5 (3)]{Muenster} for a reformulation of Robert's result),  for every $\e>0$ there exists $m(\e)$ such that  every positive 
  contraction $a$ in $A\otimes \cZ$ can be $\e$-approximated by a sum $m(\e)$-commutators of elements
  of norm $\leq 1$. 
   For  a fixed $\e>0$ this property can be expressed as a statement in the theory of $A\otimes \cZ$. 
  Since $B\otimes \cZ$ does not have a unique trace, for a small enough $\e$ the corresponding assertion 
  fails in $B\otimes \cZ$ and therefore $A\otimes \cZ$ and $B\otimes \cZ$ are not elementarily equivalent. 

If $D$ is a UHF algebra then $A\otimes D$ is monotracial
and since $D$ absorbs~$\cZ$ tensorially so does $A\otimes D$. 
Therefore the Cuntz--Pedersen  nullset of $A\otimes D$ is definable, and the proof 
proceeds as in the case of $\cZ$. 
   \end{proof} 

A \cstar-algebra $A$ is \emph{$\cZ$-absorbing} if $A\otimes \cZ\cong A$.  This is a remarkably strong and important regularity property of nuclear \cstar-algebras (see \cite{winter2017structure}). Also,~$\cZ$ itself is $\cZ$-absorbing (\cite{JiangSu}),   and being $\cZ$-absorbing is (among separable \cstar-algebras) axiomatizable (\cite[Theorem~2.5.2 (21)]{Muenster}). I am inclined to conjecture that tensoring with a fixed UHF algebra (or any nuclear \cstar-algebra)  preserves elementary equivalence among  $\cZ$-absorbing \cstar-algebras, and that tensoring with the Cuntz algebra $\cO_2$ preserves elementary equivalence.

\section{Applications to \cstar-algebras} 
\label{S.Sato}

\begin{proof}[Proof of Theorem~\ref{T.C*-P-point}] Fix a P-point $\cU$ and a separable \cstar-algebra $B$. 

\eqref{1.T.C*} 
It suffices to prove $\pi_\cU[B^\infty\cap B']\supseteq B^\cU\cap \pi_\cU[B]'$, since the other inclusion is automatic.  	
Fix $c\in B^\cU\cap \pi_\cU[B]'$. In the structure $(B^\infty, B^\cU, \pi_\cU)$ consider the 1-type of an element of $B^\infty$ consisting of the conditions $\pi_\cU(x)=c$ and $\|[x,b_n]\|=0$, for a fixed dense subset $\{b_n\mid n\in \bbN\}$ of the unit ball of~$B$. In order to see that this type is satisfiable, fix $m\geq 1$. Then the set $X=\{j\mid \max_{n\leq m} \|[c_j,b_n]\|\leq 1/m\}$ belongs to $\cU$. Define $d\in B^\infty$  by its representing sequence,  $d_j=c_j$ if $j\in X$ and $d_j=0$ if $j\notin X$. Then  $\pi_\cU(d)=c$ and $\max_{n\leq m} \|[d,b_n]\|\leq 2/m$. Since $m$ was arbitrary, the type is satisfiable. By the countable saturation of $(B^\infty,B^\cU,\pi_\cU)$  (Theorem~\ref{T.CS}), it is realized by some $c'\in B^\infty$. Clearly $c'\in B^\infty \cap B'$ and $\pi_\cU(c')=c$. Since $c$ was an arbitrary element of $B^\cU\cap \pi_\cU[B]'$, this completes the proof. 

The proof of \eqref{2.T.C*}, that $\pi_\cU[B^\infty\cap A']=B^\cU\cap \pi_\cU[A]'$ for every separable \cstar-subalgebra $A$ of $B^\infty$,   is analogous to that of \eqref{1.T.C*} and therefore omitted. 

\eqref{4.T.C*} Recall that an ideal $I$ in a \cstar-algebra is a  \emph{$\sigma$-ideal} if for every countable $X\subseteq J$ there exists a positive contraction $a\in J$ such that $ab=ba=b$ for all $b\in X$.  Fix a countable subset $\{a_n\mid n\in \bbN\}$ of $c_\cU(B)$. Consider the type of an element of $B^\infty$ with conditions $x\leq 0$, $\|x\|=1$, $\|x a_n -a_n\|=0$, and $\pi_\cU(x)=0$.  Since $c_\cU(B)$ has an approximate unit (every \cstar-algebra does; see e.g., \cite[\S 1.8]{Fa:STCstar}), this type is satisfiable. By countable saturation of $(B^\infty,B^\cU,\pi_\cU)$ it is realized by some $c$ in $B^\infty$. Then $c$ is as required.

For the remainder of the proof, we do not assume that $\cU$ is a P-point. 

\eqref{1.T.C*} implies \eqref{3.T.C*}: We need to prove that if $\cU$ is not a P-point and $B$ is a UHF algebra then $\pi_\cU[B^\infty\cap B']\not\supseteq B^\cU\cap \pi_\cU[B]'$.  This is similar to the proof of the analogous part of Theorem~\ref{T.split.P-point}.  Suppose that $\cU$ is not a P-point and fix $X_n\in \cU$ for $n\in \bbN$ such that for every $X\in \cU$ the set $X\setminus X_n$ is infinite for some $n\in \bbN$.
In $B$, identified with $\bigotimes_\bbN M_2(\bbC)$, we can choose unitaries $u_j$ and $v_j$ for all $j$ such that $\lim_{j\to \infty} \|[a,u_j]\|=0$ for all $a\in B$ but $\|[u_j,v_n]\|=2$ if $n\geq j$ (the construction is similar to that in the  proof of \cite[Lemma~3.2]{FaHaSh:Model1}, where the analogous statement for the tracial norm was proven). 
Define $c\in B^\cU$ by its representing sequence $c_j=u_n$ if $j\in X_n\setminus X_{n+1}$ (with $X_{-1}=\bbN$). Then $c\in B^\cU\cap B'$. If $b\in B^\infty$ is such that $\pi_\cU(b)=c$, then $X=\{n\mid \|b_n-c_n\|<1/4\}$ belongs to $\cU$. Let $n$ be large enough to have $X\setminus X_n$ infinite. Then $\|[v_n,b]\|\geq \|[v_n,c]\|-\|v_n\|/2>0$, and therefore $b\notin B^\infty \cap B'$. 
Since $b$ was an arbitrary $\pi_\cU$-preimage of $c$, this completes the proof.  

\eqref{5.T.C*} clearly implies \eqref{1.T.C*}, and we have already proved that \eqref{3.T.C*} implies \eqref{5.T.C*}. This completes the proof. 
\end{proof}

\section{Concluding remarks} 
\label{S.Concluding} 
Our main goal was to prove Theorems~\ref{T.AaCh}, \ref{T.A}, \ref{T.split.P-point}, \ref{T.transfer}, \ref{C.infty},  \ref{T.C*-P-point},  and~\ref{T.5.1}. Theorem~\ref{T.FV++} and Theorem~\ref{T.CS} were therefore stated and proved in what appear to be special cases of more general results on multi-sorted structures consisting of various reduced powers of a single structure and quotient maps between them. This line of research may merit further attention. %It would also be nice to see the  characterization of those~$B$ such that the quotient map $\pi_\cU$ has a right inverse provably in ZFC (see Example~\ref{Ex.Z/2Z}).  

Continuous fields of \cstar-algebras are well-studied objects (see \cite[\S IV.1.6]{Black:Operator}), and a theory of continuous fields of metric structures will inevitably be developed. Can the conclusion of Corollary~\ref{C.elementary}  be extended to the  assertion that taking  continuous fields of metric structures over  the Cantor space preserves elementary equivalence? Or, what sort of a  Feferman--Vaught--Ghasemi-style theorem holds for continuous fields of metric structures? (See e.g., \cite[Theorem~2.1]{burris1979sheaf}
for a discrete version of the desired results.) 
It should be noted that \cite[Proposition~3.10.3]{Muenster} implies that $A\prec B$ (this stands for `$A$ is an elementary submodel of $B$') does not imply $C([0,1],A)\prec C([0,1],B)$, and therefore even taking trivial continuous fields of metric structures does not necessarily preserve elementary equivalence. 

Another class of massive (albeit typically not countably saturated---see \cite[Exercise~16.8.36]{Fa:STCstar}) quotient \cstar-algebras derived from continuous fields deserves attention of model-theorists, and I'll use this opportunity to mention them. If $B$ is a separable \cstar-algebra, consider the algebra of all bounded continuous functions from $[0,\infty)$ into $B$, $C_b([0,\infty), B)$. The quotient of this algebra over the ideal $C_0([0,\infty),B)$ provides the setting for the  E-theory (\cite[\S  25]{blackadar1998k}).

Tracial von Neumann algebras give an example of an axiomatizable category not closed under the operation of taking the reduced power. This is because the asymptotic sequence algebra $B^\infty$ is infinite-dimensional and countaby saturated, and therefore not a von Neumann algebra. Therefore the following quotient merits our attention.

\begin{example}\label{Ex.Rinfty}
Let $R$ denote the hyperfinite II$_1$ factor and let $\tau$ denote its unique tracial state. 
Then $(R,\tau)$ is naturally construed as a metric structure with respect to the $\ell_2$-norm, $\|a\|_2:=\sqrt{\tau(a^*a)}$ (see \cite[\S 3.2]{FaHaSh:Model2}).  	Since the category of II$_1$ factors is axiomatizable, the ultrapower of $R$ in this language (also known as the \emph{tracial ultrapower}),~$R^\cU$, is a II$_1$ factor.\footnote{This fact has been known long before general ultrapowers of metric structures had been defined; see the introduction to \cite{FaHaSh:Model2}.} The reduced power $R^\infty$ is however not a II$_1$ factor, or even a von Neumann algebra. One reason for this is that in a von Neumann algebra every directed family of positive contractions has a supremum, and this is not the case with the reduced powers associated with the Fr\'echet ideal. This reduced power belongs to the class of \cstar-algebras  extensively studied in recent years (see \cite{castillejos2018nuclear}, and~\cite{goldbring2018correspondences} for their model-theoretic analysis). 
\end{example}

The structure $R^\infty$  was recently used in~\cite{CGSTW}. By the following.  $R^\infty$ is for all practical purposes isomorphic to an ultrapower of~$KR$ (where $K$ is the functor defined in \S\ref{S.FunctorK} and  $KR$ is embedded into $R^\infty$ by using Proposition~\ref{P.embedding}). 

\begin{prop} Consider $R$, $KR$,  and $R^\infty$ as metric structures with respect to the $\ell_2$-norm.  Then $R^\infty$ is a countably saturated elementary extension of $KR$. 
\end{prop} 

\begin{proof} The first part follows from the  first part of Proposition~\ref{P.elementary}  and the second part is a consequence of \cite[Theorem~2.1]{FaSh:Rigidity} (see also \cite[Theorem~16.5.1]{Fa:STCstar}). 
\end{proof}

 \bibliographystyle{amsplain}
\bibliography{ifmainbib}

\providecommand{\bysame}{\leavevmode\hbox to3em{\hrulefill}\thinspace}
\providecommand{\MR}{\relax\ifhmode\unskip\space\fi MR }
\providecommand{\MRhref}[2]{%
  \href{http://www.ams.org/mathscinet-getitem?mr=#1}{#2}
}
\providecommand{\href}[2]{#2}
\begin{thebibliography}{10}

\bibitem{adams1991amenability}
S.~Adams and R.~Lyons, \emph{Amenability, {K}azhdan's property and percolation
  for trees, groups and equivalence relations}, Israel J. Math. \textbf{75}
  (1991), no.~2-3, 341--370.

\bibitem{BYBHU}
I.~Ben~Yaacov, A.~Berenstein, C.W. Henson, and A.~Usvyatsov, \emph{Model theory
  for metric structures}, Model Theory with Applications to Algebra and
  Analysis, Vol. II (Z.~Chatzidakis et~al., eds.), London Math. Soc. Lecture
  Notes Series, no. 350, London Math. Soc., 2008, pp.~315--427.

\bibitem{blackadar1998k}
B.~Blackadar, \emph{K-theory for operator algebras}, MSRI Publications, vol.~5,
  Cambridge University Press, 1998.

\bibitem{Black:Operator}
\bysame, \emph{Operator algebras}, Encyclopaedia of Mathematical Sciences, vol.
  122, Springer-Verlag, Berlin, 2006, Theory of \cstar-algebras and von Neumann
  algebras, Operator Algebras and Non-commutative Geometry, III.

\bibitem{burris1979sheaf}
S.~Burris and H.~Werner, \emph{Sheaf constructions and their elementary
  properties}, Trans. Amer. Math. Soc. \textbf{248} (1979), no.~2, 269--309.

\bibitem{CGSTW}
J.~Carri\'on, J.~Gabe, C.~Schafhauser, A.~Tikuisis, and S.~White,
  \emph{Classifying *-homomorphisms}, Manuscript in preparation.

\bibitem{castillejos2018nuclear}
J.~Castillejos, S.~Evington, A.~Tikuisis, S.~White, and W.~Winter,
  \emph{Nuclear dimension of simple \cstar-algebras}, Inventiones mathematicae
  \textbf{224} (2021), no.~1, 245--290.

\bibitem{ChaKe}
C.~C. Chang and H.~J. Keisler, \emph{Model theory}, third ed., Studies in Logic
  and the Foundations of Mathematics, vol.~73, North-Holland Publishing Co.,
  Amsterdam, 1990.

\bibitem{ChaKe:1973}
C.C. Chang and H.J. Keisler, \emph{Model theory}, North--Holland, 1973.

\bibitem{chodounsky2019there}
D.~Chodounsk\`y and O.~Guzm\'an, \emph{There are no {P}-points in {S}ilver
  extensions}, Israel J. Math. \textbf{232} (2019), no.~2, 759--773.

\bibitem{Connes:Class}
A.~Connes, \emph{Classification of injective factors. \text{Cases}
  $\text{II}_1$, $\text{II}_\infty$, $\text{III}_\lambda,$ $\lambda \neq1$},
  Ann. of Math. (2) \textbf{104} (1976), 73--115.

\bibitem{effros1978c}
E.~G. Effros and J.~Rosenberg, \emph{C*-algebras with approximately inner
  flip}, Pacific J. Math \textbf{77} (1978), no.~2, 417--443.

\bibitem{elliott2015classification}
G.A. Elliott, G.~Gong, H.~Lin, and Z.~Niu, \emph{On the classification of
  simple amenable \cstar-algebras with finite decomposition rank, {II}}, arXiv
  preprint arXiv:1507.03437 (2015).

\bibitem{Fa:Embedding}
I.~Farah, \emph{Embedding partially ordered sets into $\omega^\omega$}, Fund.
  Math. \textbf{151} (1996), 53--95.

\bibitem{farah2019between}
\bysame, \emph{Between ultrapowers and asymptotic sequence algebras}, arXiv
  preprint arXiv:1904.11776 (2019).

\bibitem{Fa:STCstar}
\bysame, \emph{Combinatorial set theory and \cstar-algebras}, Springer
  Monographs in Mathematics, Springer, 2019.

\bibitem{Muenster}
I.~Farah, B.~Hart, M.~Lupini, L.~Robert, A.~Tikuisis, A.~Vignati, and
  W.~Winter, \emph{Model theory of \cstar-algebras}, Memoirs AMS (to appear).

\bibitem{FaHaSh:Model1}
I.~Farah, B.~Hart, and D.~Sherman, \emph{Model theory of operator algebras {I}:
  Stability}, Bull. London Math. Soc. \textbf{45} (2013), 825--838.

\bibitem{FaHaSh:Model2}
\bysame, \emph{Model theory of operator algebras {II}: Model theory}, Israel J.
  Math. \textbf{201} (2014), 477--505.

\bibitem{FaSh:Rigidity}
I.~Farah and S.~Shelah, \emph{Rigidity of continuous quotients}, J. Math. Inst.
  Jussieu \textbf{15} (2016), no.~01, 1--28.

\bibitem{farah2020between}
\bysame, \emph{Between reduced powers and ultrapowers, {II}}, arXiv preprint
  arXiv:2011.07352 (2020).

\bibitem{feferman1959first}
S.~Feferman and R.~Vaught, \emph{The first order properties of products of
  algebraic systems}, Fundamenta Mathematicae \textbf{47} (1959), no.~1,
  57--103.

\bibitem{gabe2019new}
J.~Gabe, \emph{A new proof of {K}irchberg's {$\cO_2$}-stable classification},
  J. Reine Angew. Math (Crelle's Journal) (to appear).

\bibitem{ghasemi2016reduced}
S.~Ghasemi, \emph{Reduced products of metric structures: a metric
  {F}eferman--{V}aught theorem}, J. Symb. Log. \textbf{81} (2016), no.~3,
  856--875.

\bibitem{ghasemi2014reduced}
\bysame, \emph{Reduced products of metric structures: a metric
  {F}eferman--{V}aught theorem}, J. Symbolic Logic (to appear).

\bibitem{goldbring2018correspondences}
I.~Goldbring, B.~Hart, and T.~Sinclair, \emph{Correspondences, ultraproducts
  and model theory}, arXiv preprint arXiv:1809.00049 (2018).

\bibitem{JiangSu}
X.~Jiang and H.~Su, \emph{On a simple unital projectionless {C}$^*$-algebra},
  Amer. J. Math \textbf{121} (1999), 359--413.

\bibitem{Kirc:Central}
E.~Kirchberg, \emph{Central sequences in \cstar-algebras and strongly purely
  infinite algebras}, Operator Algebras: The Abel Symposium 2004, Abel Symp.,
  vol.~1, Springer, Berlin, 2006, pp.~175--231.

\bibitem{KirRo:Central}
E.~Kirchberg and M.~R{\o}rdam, \emph{Central sequence {C}*-algebras and
  tensorial absorption of the {J}iang--{S}u algebra}, J. Reine Angew. Math.
  \textbf{695} (2014), 175--214.

\bibitem{kirchberg2014central}
\bysame, \emph{Central sequence \cstar-algebras and tensorial absorption of the
  {J}iang--{S}u algebra}, J. Reine Angew. Math. \textbf{2014} (2014), no.~695,
  175--214.

\bibitem{Mark:Model}
D.~Marker, \emph{Model theory}, Graduate Texts in Mathematics, vol. 217,
  Springer-Verlag, New York, 2002.

\bibitem{McDuff:Central}
D.~McDuff, \emph{Central sequences and the hyperfinite factor}, Proc. London
  Math. Soc. \textbf{21} (1970), 443--461.

\bibitem{Olin:Direct}
P.~Olin, \emph{Direct multiples and powers of modules}, Fund. Math. \textbf{73}
  (1971/72), no.~2, 113--124.

\bibitem{Oza:Dixmier}
N.~Ozawa, \emph{Dixmier approximation and symmetric amenability for
  \cstar-algebras}, J. Math. Sci. Univ. Tokyo \textbf{20} (2013), 349--374.

\bibitem{Phi:Exponential}
N.C. Phillips, \emph{Exponential length and traces}, Proc. Roy. Soc. Edinburgh
  Sect. A \textbf{125} (1995), no.~1, 13--29.

\bibitem{Phi:Classification}
\bysame, \emph{A classification theorem for nuclear purely infinite simple
  {\cstar}-algebras}, Doc. Math. \textbf{5} (2000), 49--114.

\bibitem{Rob:Nuclear}
L.~Robert, \emph{Nuclear dimension and sums of commutators}, Indiana University
  Mathematics Journal (2015), 559--576.

\bibitem{Ror:Classification}
M.~R{\o}rdam, \emph{Classification of nuclear {\cstar}-algebras}, Encyclopaedia
  of Math. Sciences, vol. 126, Springer-Verlag, Berlin, 2002.

\bibitem{Ru}
W.~Rudin, \emph{Homogeneity problems in the theory of \v{C}ech
  compactifications}, Duke Mathematics Journal \textbf{23} (1956), 409--419.

\bibitem{sato2011discrete}
Y.~Sato, \emph{Discrete amenable group actions on von {N}eumann algebras and
  invariant nuclear \cstar-subalgebras}, arXiv preprint arXiv:1104.4339 (2011).

\bibitem{schafhauser2018subalgebras}
C.~Schafhauser, \emph{Subalgebras of simple {AF}-algebras}, Ann. Math. (2) (to
  appear).

\bibitem{shelah1990classification}
S.~Shelah, \emph{Classification theory and the number of non-isomorphic
  models}, Elsevier, 1990.

\bibitem{Sh:PIF}
\bysame, \emph{Proper and improper forcing}, Perspectives in Mathematical
  Logic, Springer, 1998.

\bibitem{shelah2003logical}
\bysame, \emph{Logical dreams}, Bull. Amer. Math. Soc. \textbf{40} (2003),
  no.~2, 203--228.

\bibitem{sherman2009notes}
D.~Sherman, \emph{Notes on automorphisms of ultrapowers of {II$_1$} factors},
  Studia Mathematica \textbf{195} (2009), no.~3, 201--217.

\bibitem{winter2017structure}
W.~Winter, \emph{Structure of nuclear \cstar-algebras: From quasidiagonality to
  classification, and back again}, Proceedings of the 2018 ICM, Rio de Janeiro
  \textbf{2} (2017), 1797--1820.

\end{thebibliography}

\end{document}